\documentclass[final]{siamltex}
\usepackage{amsmath,amsfonts,amssymb}
\usepackage{latexsym}
\usepackage{graphicx,overpic}
\usepackage{subfigure,caption}
\usepackage{pgfplots}
\pgfplotsset{compat=newest}
\usepackage{color}
\usepackage{booktabs}

\newcommand{\R}{\mathbb R}
\newcommand{\bA}{\mathbf A}

\newcommand{\bH}{\mathbf H}
\newcommand{\bI}{\mathbf I}

\newcommand{\bP}{\mathbf P}

\newcommand{\bV}{\mathbf V}

\newcommand{\bg}{\mathbf g}
\newcommand{\bn}{\mathbf n}
\newcommand{\be}{\mathbf e}

\newcommand{\bu}{\mathbf u}
\newcommand{\bU}{\mathbf U}
\newcommand{\bv}{\mathbf v}
\newcommand{\bw}{\mathbf w}

\newcommand{\bx}{\mathbf x}

\newcommand{\bbf}{\mathbf f}
\newcommand{\A}{\mathcal A}

\newcommand{\T}{\mathcal T}

\newcommand{\divG}{{\mathop{\,\rm div}}_{\Gamma}}
\newcommand{\gradG}{\nabla_{\Gamma}}

\newcommand{\cT}{\mathcal T}

\newcommand{\OGamma}{\Omega^\Gamma_h}

\renewcommand{\div}{\textrm{div}\ \!}

\newcommand{\tr}{{\rm tr}}
\DeclareGraphicsExtensions{.pdf,.eps,.ps,.eps.gz,.ps.gz,.eps.Y}

\newcommand{\la}{\left\langle}
\newcommand{\ra}{\right\rangle}

\newtheorem{assumption}{Assumption}[section]

\newtheorem{remark}{Remark}[section]
\begin{document}
\title{A Trace Finite Element Method  for Vector-Laplacians on Surfaces}
\author{Sven Gro{\ss}\thanks{Institut f\"ur Geometrie und Praktische  Mathematik, RWTH-Aachen
University, D-52056 Aachen, Germany (gross@igpm.rwth-aachen.de).}
\and
Thomas Jankuhn\thanks{Institut f\"ur Geometrie und Praktische  Mathematik, RWTH-Aachen
University, D-52056 Aachen, Germany (jankuhn@igpm.rwth-aachen.de).}
\and Maxim A. Olshanskii\thanks{Department of Mathematics, University of Houston, Houston, Texas 77204-3008 (molshan@math.uh.edu); Partially supported by NSF through the Division of Mathematical Sciences grant 1717516.}
\and Arnold Reusken\thanks{Institut f\"ur Geometrie und Praktische  Mathematik, RWTH-Aachen
University, D-52056 Aachen, Germany (reusken@igpm.rwth-aachen.de).}
}
\maketitle
%\tableofcontents
%
\begin{abstract}
 We consider a vector-Laplace problem   posed on a 2D surface embedded in a 3D domain, which results from the modeling of surface fluids based on exterior Cartesian differential operators. The main topic of this paper is the development and analysis of a  finite element method  for the discretization of this surface partial differential equation. We apply the trace finite element technique,  in which finite element spaces  on a background shape-regular tetrahedral mesh that is  surface-independent are used for discretization.
In order  to satisfy  the constraint that the solution vector field is tangential to the surface we introduce  a Lagrange multiplier. We  show  well-posedness of the resulting saddle point formulation.  A discrete variant of this formulation is introduced which contains suitable stabilization terms and is based on trace finite element spaces. For this method we derive optimal discretization error bounds. Furthermore  algebraic properties of the resulting discrete saddle point problem are studied. In particular an optimal Schur complement preconditioner is proposed. Results of a numerical experiment are included.
\end{abstract}
%\begin{AMS} 65M06, 65D05
% \end{AMS}

\section{Introduction}
Fluid equations  on manifolds appear in the literature on  mathematical modeling  of emulsions, foams and  biological  membranes, e.g.~\cite{scriven1960dynamics,slattery2007interfacial,arroyo2009,brenner2013interfacial,rangamani2013interaction,rahimi2013curved}.
We refer the reader to the recent contributions~\cite{Jankuhn1,Gigaetal} for  derivations of governing surface Navier--Stokes equations in terms of  exterior Cartesian differential operators for the general case of a viscous incompressible material  surface, which is embedded in 3D and may evolve in time. This Navier--Stokes and other models of viscous fluidic surfaces or interfaces  involve  the \emph{vector-Laplace operator} treated in this paper.  We note that there are different definitions of surface vector-Laplacians, cf. Remark~\ref{remLaplacian}. In this paper we treat a vector-Laplace problem that results from the modeling of surface fluids based on exterior Cartesian differential operators.

Several important properties of surface fluid equations, such as existence, uniqueness and regularity of weak solutions, their continuous dependence on initial data and a relation of these equations to the problem of finding geodesics on the group of volume preserving diffeomorphisms have been studied in the  literature,
%Several important properties of surface fluid equations have been derived in the (mathematical) literature,
e.g.,~\cite{ebin1970groups,Temam88,taylor1992analysis,arnol2013mathematical,mitrea2001navier,arnaudon2012lagrangian,miura2017singular}. Concerning the development and analysis of numerical  methods for surface fluid equations there are  very few papers, e.g., \cite{barrett2014stable,nitschke2012finite,reuther2015interplay,holst2012geometric,hansbo2016analysis,ReuskenZhang} and research on this topic has started only recently.  Much more is known on discretization methods for \emph{scalar} elliptic and parabolic PDEs on surfaces; see the review of surface finite element methods in~\cite{DEreview,olshanskii2016trace}.

In this paper we introduce and analyze  a finite element method for the numerical solution of a vector-Laplace problem   posed on a 2D surface embedded in a 3D domain. The approach developed here benefits from the embedding of the surface in $\mathbb{R}^3$ and uses elementary tangential calculus  to formulate the equations in terms of exterior differential operators in Cartesian coordinates. Following this paradigm, the finite element spaces we use are also tailored to an ambient background mesh. This mesh is surface-independent and consists of shape-regular tetrahedra.
As in previous work on scalar elliptic and parabolic surface PDEs (cf. the overview paper \cite{olshanskii2016trace}) we use the \emph{trace} of such an outer finite element space for the discretization of the  vector-Laplace problem.  Hence, the method that we present is a special unfitted finite element method.
One distinct difficulty of applying (both fitted and unfitted) finite element methods to surface vector-Laplace and surface Navier--Stokes equations is to satisfy numerically the constraint that the solution vector field $\bu_h$ is tangential to the surface $\Gamma$, i.e., $\bu_h\cdot\bn=0$, where $\bn$ is the normal vector field to $\Gamma$, cf. the discussion in Remark~\ref{rem_t}.  The method that we present  handles this constraint weakly by introducing a Lagrange multiplier. The resulting saddle point variational formulation is discretized by using standard trace finite element spaces. In the discrete variational formulation certain consistent stabilization terms are included which are essential  for discrete inf-sup stability, algebraic stability, and for the derivation of (optimal) preconditioners for the discrete problem.

The main contributions of this paper are the following. We introduce the Lagrange multiplier formulation for the continuous problem and show  well-posedness of the resulting saddle point formulation. We present a finite element variational formulation which contains suitable stabilization terms and is based on trace finite element spaces. For this method we work out an error  analysis. This analysis shows that for discrete stability and optimal discretization error bounds the background space for the Lagrange multiplier can be chosen as piecewise polynomial of the same order or one order lower as for the primal variable. A further main contribution of the paper is the analysis of algebraic properties of the resulting discrete saddle point problem. In particular we derive an optimal Schur complement preconditioner. We note that in the error analysis we do \emph{not} include geometric errors induced by an approximation $\Gamma_h \approx \Gamma$. The analysis of the effects of such geometric errors is left
for future research.

Our approach is very different from the one based on finite element exterior calculus suitable for discretizing the Hodge Laplacian on hypersurfaces; see  \cite{holst2012geometric}. In the recent paper~\cite{hansbo2016analysis} a finite element method for a similar vector-Laplace problem is studied. That method, however, uses a penalty technique instead of  a Lagrange multiplier formulation and  requires meshes fitted to the surface. Finally we note that the finite element methods for surface Navier--Stokes equations presented in \cite{nitschke2012finite,reuther2015interplay} are based on a surface curl-formulation, which is not applicable to the surface  vector-Laplace problem that we consider. We also note that no discretization error analyses are given in \cite{nitschke2012finite,reuther2015interplay}. None of these related papers have considered unfitted finite element methods.

This paper is meant to be the first one in a series of papers devoted to numerical simulation of  fluid equations on (evolving) manifolds.
Our longer-term goal is to provide efficient and reliable computational tools for the numerical solution of fluid equations on a time-dependent surface $\Gamma(t)$ including the cases when parametrization of $\Gamma(t)$ is not explicitly available and $\Gamma(t)$ may undergo topological changes. This motivates our choice to use unfitted surface-independent meshes to define finite element spaces --- a methodology that proved to work very
well for scalar PDEs posed on $\Gamma(t)$~\cite{olshanskii2014eulerian,olshanskii2014error}.

The remainder of the paper is organized as follows. We introduce in section~\ref{s_cont} the vector-Laplace model probem and  notions of  tangential differential calculus. We give a weak formulation of the problem with Lagrange multiplier and show its well-posedness. An unfitted finite element method known as the TraceFEM for the surface vector-Laplace problem is introduced in section~\ref{s_TraceFEM}. In section~\ref{s_error} an error analysis of this method is presented. We derive discrete LBB stability for certain pairs of Trace FE spaces. The main result of this section  is an optimal order error estimate in the energy norm. An optimal order discretization error estimate in the $L^2$ norm is shown in section~\ref{s_L2}. In section~\ref{s_cond} we prove that the
spectral condition number of the resulting saddle point stiffness matrix is bounded by $c h^{-2}$, with a constant  $c$ that is independent of the position of the surface $\Gamma$ relative to the underlying triangulation.  We also present an optimal Schur complement preconditioner. Numerical results in section~\ref{sectExp} illustrate  the performance of the method in terms of discretization error convergence and efficiency of the linear solver.

%^\emph{MO: I am not sure if we need this paragraph.}
%We close the introduction section noting that computational methods  for fluid equations on surfaces and numerical analysis of these methods is a relatively new field of research.
%Exploring the line of research starting from the seminal paper \cite{scriven1960dynamics}, it is noted in \cite{arroyo2009} and \cite{nitschke2012finite} that ``the equations of motion are formulated intrinsically in a two-dimensional manifold with time-varying metric and make extensive use of the covariant derivative and calculations in local coordinates, which involve the coefficients of the Riemannian connection and its derivatives. The complexity of the equations may explain why they are often written but never solved for arbitrary surfaces.'' Recent research addressing to a certain extend the numerical solution of fluid equations on surfaces includes   \cite{nitschke2012finite,rangamani2013interaction,rahimi2013curved,barrett2014stable,reuther2015interplay,rodrigues2015semi,ReuskenZhang,hansbo2016analysis,Jankuhn1}.
%We consider the present paper as another contribution to this emerging research topic.

\section{Continuous problem}\label{s_cont}
Assume that $\Gamma$ is a closed sufficiently smooth surface in $\mathbb{R}^3$. The outward pointing unit normal on $\Gamma$ is denoted by $\bn$, and the orthogonal projection on the tangential plane is given by $\bP=\bP(x):= \bI - \bn(x)\bn(x)^T$, $x \in \Gamma$.  For vector functions $\bu:\, \Gamma \to \mathbb{R}^3$ we use a constant extension from $\Gamma$ to its neighborhood $\mathcal{O}(\Gamma)$ along the normals $\bn$, denoted by $\bu^e :\,\mathcal{O}(\Gamma)\to\mathbb{R}^3$.  Note that on $\Gamma$ we have $\nabla\bu^e= \nabla (\bu\circ p)=\nabla\bu^e\bP$, with $\nabla \bu:= (\nabla u_1~ \nabla u_2 ~\nabla u_3)^T \in \mathbb{R}^{3 \times 3}$ for vector functions $\bu$ (note the transpose; this notation is usual in computational fluid dynamics).  For scalar functions $u:\, \mathcal{O}(\Gamma) \to \Bbb{R}$ the gradient $\nabla u$ denotes the column vector consisting of the partial derivatives. In the remainder  this locally unique extension $\bu^e$ to a small neighborhood  of $\Gamma$ is also
denoted by $\bu$.
On $\Gamma$ we consider the surface strain tensor \cite{GurtinMurdoch75} given by
\begin{equation} \label{strain}
 E_s(\bu):= \frac12 \bP (\nabla \bu +\nabla \bu^T)\bP = \frac12(\nabla_\Gamma \bu + \nabla_\Gamma \bu^T), \quad \nabla_\Gamma \bu:= \bP \nabla \bu \bP.
 \end{equation}
We also use the surface divergence operators for  $\bu: \Gamma \to \R^3$ and $\bA: \Gamma \to \mathbb{R}^{3\times 3}$.  These are defined as follows:
\begin{align*}
 \divG \bu & := \tr (\gradG \bu)= \tr (\bP (\nabla\bu) \bP)=\tr (\bP (\nabla\bu))=\tr ((\nabla\bu) \bP),  \\
% \divG \bA  & := \begin{pmatrix} \divG (e_1^T \bA) \\
%               \vdots \\
%               \divG (e_n^T \bA)
%              \end{pmatrix}
 \divG \bA  & := \left( \divG (\be_1^T \bA),\,
               \divG (\be_2^T \bA),\,
               \divG (\be_3^T \bA)\right)^T,
\end{align*}
with $\be_i$ the $i$th basis vector in $\R^3$.
For a given  force vector $\mathbf{f} \in L^2(\Gamma)^3$, with $\mathbf{f}\cdot\bn=0$, we consider the following elliptic partial differential equation: determine $\bu:\, \Gamma \to \R^3$ with $\bu\cdot\bn =0$  and
\begin{equation} \label{strongform}
 - \bP \divG (E_s(\bu)) + \bu=\mathbf{f} \quad \text{on}~\Gamma.
\end{equation}
We added the zero order term $\bu$ on the left-hand side in \eqref{strongform} to avoid technical details related to the kernel of the strain tensor $E_s$.
\begin{remark} \label{remLaplacian} \rm
In this paper we  consider the operator $ \bP \divG \circ E_s$ because it is a key component in the modeling of Newtonian surface fluids and fluidic membranes \cite{scriven1960dynamics,GurtinMurdoch75,barrett2014stable,Gigaetal,Jankuhn1}.
We note that in the literature there are different formulations of the surface Navier--Stokes equations, and some of these are formally obtained by substituting Cartesian differential operators by their geometric counterparts~\cite{Temam88,cao1999navier} rather than from  first mechanical principles. These formulations may involve different surface Laplace type operators.
%Generally speaking, this \textbf{\small MO: Any other applications?} stimulates an interest in numerical methods for different surface vector-Laplace operators. For example,
In the recent preprint \cite{hansbo2016analysis} the  Bochner (also called rough) Laplacian $\bu \to \Delta_\Gamma\bu:=\bP \divG (\gradG \bu)$ is treated numerically. Another Laplacian operator, which in a natural way arises in differential geometry and exterior calculus is the so-called Hodge Laplacian.
The diagram below (from \cite{Jankuhn1}) and identities ~\eqref{LaplacianALL} illustrate some `correspondences' between Cartesian and different surface operators. For $\bu$ on $\Gamma$ we assume $\bu\cdot\bn=0$.
\[
\begin{array}{rcccccc}
\text{\rm In}~~\mathbb{R}^{3}\,: &-\div(\nabla\bu+\nabla^T\bu)& \overset{\div\bu=0}{=}  & -\Delta\bu& {=}  & (\mbox{rot}^T\mbox{rot}-\nabla\div)\bu\\
 &\wr& &\wr& &\wr\\
 \text{On $\Gamma$}\,: & \underbrace{-\bP \divG (2 E_s(\bu))}& \overset{\divG\bu=0}{\neq}  &  \underbrace{-\Delta_\Gamma\bu}& {\neq}  & \underbrace{-\Delta_\Gamma^H\bu}\\
 &\text{surface}& &\text{Bochner }& &\text{Hodge}\\
 &\text{diffusion}& &\text{Laplacian}& &\text{Laplacian}\\
\end{array}
\]
For a smooth surface $\Gamma\subset\mathbb{R}^3$ with Gauss curvature $K$ we have, cf.  Lemma~2.1 in~\cite{Jankuhn1} and the Weitzenb\"{o}ck identity~\cite{rosenberg1997laplacian}, the following equalities for a tangential field $\bu$ :
\begin{equation}\label{LaplacianALL}
-\bP \divG (2E_s(\bu))=-\Delta_\Gamma\bu-K\bu=-\Delta_\Gamma^H\bu-2K\bu\quad\text{on}~\Gamma,
\end{equation}
where for the first equality to hold, $\bu$ should satisfy $\divG\bu=0$.
\end{remark}
\ \\[1ex]

For the weak formulation of this vector-Laplace problem we use the space $\bV:=H^1(\Gamma)^3$, with norm
 \begin{equation} \label{H1norm}
  \|\bu\|_{1}^2:=\int_{\Gamma}\|\bu(s)\|_2^2 + \|\nabla\bu (s)\|_2^2\,ds,
 \end{equation}
where $\|\cdot\|_2$ denotes the vector and matrix $2$-norm. Note that due to $\nabla \bu=\nabla\bu^e= \nabla \bu\bP$ on $\Gamma$ only {tangential} derivatives are included in this $H^1$-norm. The corresponding  space of tangential vector fields is denoted by
\begin{equation}   \label{defVT}
 \bV_T:= \{\, \bu \in \bV~|~ \bu\cdot \bn =0 \quad \text{on}~~\Gamma\,\}.
\end{equation}
For $\bu \in \bV$ we use the following notation for the orthogonal decomposition into tangential and normal parts:
\[ \bu = \bu_T + u_N\bn,\quad \bu_T\cdot\bn=0.
\]
We introduce the bilinear form
\begin{align*}
a(\bu,\bv)& :=  \int_\Gamma E_s(\bu):E_s(\bv) + \bu \cdot \bv \, ds= \int_\Gamma {\rm tr}\big(E_s(\bu) E_s(\bv)\big) + \bu \cdot \bv \, ds, ~~\bu,\bv \in \bV.
\end{align*}
For given  $\mathbf{f}$ as above
%(Note that in the definition of the strain tensor $E_s(\bu)$ as in \eqref{strain} we need the constant extension $\bu^e$ of $\bu \in V$, similarly for $\divG \bu$).
we consider the following variational formulation of \eqref{strongform}: determine
$\bu=\bu_T \in \bV_T$ such that
\begin{equation} \label{vectLaplace}
 a(\bu_T,\bv_T)= (\mathbf{f},\bv_T)_{L^2(\Gamma)} \quad \text{for all}~~\bv_T \in \bV_T.
\end{equation}
The bilinear form $a(\cdot,\cdot)$ is continuous on $\bV_T$. Ellipticity of $a(\cdot,\cdot)$ on $\bV_T$ follows from the following surface Korn inequality, which is derived in \cite{Jankuhn1}.
\begin{lemma} \label{Kornlemma}
Assume $\Gamma$ is $C^2$ smooth. There exists $c_K >0$ such that
 \begin{equation} \label{korn}
\|\bu\|_{L^2(\Gamma)}+ \|E_s(\bu)\|_{L^2(\Gamma)} \geq c_K \|\bu\|_{1} \quad \text{for all}~~\bu \in \bV_T.
 \end{equation}
\end{lemma}
Hence, the weak formulation \eqref{vectLaplace} is a well-posed problem. The unique solution is denoted by $\bu^\ast=\bu_T^\ast$.

\begin{remark}\label{rem_t}\rm
The weak formulation \eqref{vectLaplace} is not very suitable for a finite element Galerkin discretization, because we then need finite element functions that are \emph{tangential} to $\Gamma$, which are not easy to construct.
If $\Gamma\cap K$ is curved in a simplex $K$ where $\bu_h$ is polynomial, then it is easy to see that enforcing $\bu_h\cdot\bn=0$ on $\Gamma\cap K$ may lead to `locking', i.e. only $\bu_h=0$ satisfies the constraint.
Alternatively, one can approximate a smooth manifold $\Gamma$ by a polygonal surface $\Gamma_h$ (in practice, this is often done  for the purpose of numerical integration; moreover, only $\Gamma_h$ is available if finding the position of the surface is  part of the problem). In this case the surface $\Gamma_h$ has a \textit{discontinuous} normal field $\bn_h$ and enforcing the tangential constraint, $\bu_h\cdot\bn_h=0$ on $\Gamma_h$, for a \textit{continuous} finite element vector field $\bu_h$ may lead to a locking effect as well.
\end{remark}
\ \\[1ex]
In view of the remark above we introduce, in the same spirit as in \cite{hansbo2016stabilized,hansbo2016analysis,Jankuhn1}, a weak formulation in a space that is larger than $\bV_T$ and which allows nonzero normal components in the surface vector fields. However, different from the approach  used in these papers we treat the tangential condition with the help of a Lagrange multiplier.  The following basic relation will be very useful:
\begin{equation} \label{idfund}
E_s(\bu)=E_s(\bu_T) + u_N \bH,
\end{equation}
where $\bH := \nabla \bn$ is the shape operator (second fundamental form) on $\Gamma$.
We introduce the following
Hilbert space:
\[
\bV_\ast  :=\{\, \bu \in L^2(\Gamma)^3\,:\,\bu_T \in \bV_T,~u_N\in L^2(\Gamma)\,\},  \quad\text{with}~~
\|\bu\|_{V_\ast}^2:=\|\bu_T\|_{1}^2+\|u_N\|_{L^2(\Gamma)}^2.
\]
Note that $\bV_\ast\sim \bV_T \oplus L^2(\Gamma)$.
Based on the identity \eqref{idfund} we introduce, with some abuse of notation, the bilinear form
\begin{equation} \label{defaalt}
a(\bu,\bv) :=  \int_\Gamma {\rm tr}\big((E_s(\bu_T)+u_N\bH)( E_s(\bv_T)+ v_N \bH)\big) +\bu \cdot \bv \, ds, \quad \bu,\bv \in \bV_\ast.
\end{equation}
This bilinear form is well-defined and continuous on $\bV_\ast$.  We enforce the condition $\bu\in \bV_T$ with the help of a Lagrange multiplier.
For given $\bg \in L^2(\Gamma)^3$ (note that we allow $\bg$ not necessarily tangential) we  introduce the following saddle point problem: determine $(\bu,\lambda) \in \bV_\ast \times L^2(\Gamma)$ such that
 \begin{equation} \label{weak1a} \begin{aligned}
           a(\bu,\bv) +(\bv\cdot \bn,\lambda)_{L^2(\Gamma)} &=(\mathbf{g},\bv)_{L^2(\Gamma)} &\quad &\text{for all }\bv \in \bV_\ast, \\
           (\bu\cdot\bn,\mu)_{L^2(\Gamma)} & = 0 &\qquad &\text{for all }\mu \in L^2(\Gamma).
                                      \end{aligned}
 \end{equation}

Well-posedness of this saddle point problem is derived in the following theorem.
\begin{theorem}
 The problem \eqref{weak1a} is well-posed. Its unique solution $(\bu^\ast,\lambda) \in \bV_\ast \times L^2(\Gamma)$ has the following properties:
 \begin{align}
   1.\quad&\bu^\ast \cdot \bn  =0,\\
    2.\quad&\bu_T^\ast =\bu_T,~~  \text{where $\bu_T$ is the unique solution of \eqref{vectLaplace} with $\mathbf{f}:=\bg_T=\bP \bg$}, \label{hhj} \\
3.\quad&\lambda  = g_N- \tr \big(E_s(\bu_T^\ast) \bH)\big),~~ \text{for}~\bg=\bg_T+g_N\bn. \label{charlambda}
\end{align}
\end{theorem}
\begin{proof}
 Note that $\bv \in \bV_\ast$ satisfies  $(\bv\cdot\bn,\mu)_{L^2(\Gamma)}=0$ for all $\mu \in L^2(\Gamma)$ iff $\bv \in \bV_T$. From this and \eqref{korn} it follows that
 $a(\cdot,\cdot)$ is elliptic on $\bV_T$, the subspace of $\bV_\ast$
consisting of all functions $\bu$ that satisfy the second equation in \eqref{weak1a}. The multiplier bilinear form $(\bv,\mu) \mapsto (\bv\cdot \bn,\mu)_{L^2(\Gamma)}$ has the inf-sup property
\[
\inf_{\mu \in L^2(\Gamma)} \sup_{\bv \in \bV_\ast}  \frac{(\bv\cdot \bn,\mu)_{L^2(\Gamma)}}{\|\bv\|_{V_\ast}\|\mu\|_{L^2(\Gamma)}} \geq \inf_{\mu \in L^2(\Gamma)} \sup_{v_N \in L^2(\Gamma)} \frac{(v_N,\mu)_{L^2(\Gamma)}}{\|v_N\|_{L^2(\Gamma)}\|\mu\|_{L^2(\Gamma)}}=1.
\]
Furthermore, the bilinear forms are continuous. Hence, we have a well-posed saddle point formulation, with a unique solution denoted by $(\bu^\ast,\lambda)$. From the second equation in \eqref{weak1a} one obtains $\bu^\ast \cdot \bn=0$. If in the first equation we restrict to $\bv=\bv_T \in V_T$ we see that $ \bu_T^\ast$ satisfies the same variational problem as in \eqref{vectLaplace} with $\mathbf{f}:=\bP \bg$, hence, \eqref{hhj} holds.

If in the first equation in \eqref{weak1a} we take $\bv=v_N \bn$ and use $ \bu^\ast \cdot \bn=0$ we get
\[ \begin{split}
  (\lambda, v_N)_{L^2(\Gamma)} & = (\bg, v_N \bn)_{L^2(\Gamma)}- a( \bu^\ast, v_N\bn) \\  & =(g_N, v_N)_{L^2(\Gamma)} - \int_{\Gamma} \tr \big(E_s(\bu_T^\ast) E_s(v_N \bn)\big) \, ds \\ & =(g_N, v_N)_{L^2(\Gamma)} -(\tr \big(E_s(\bu_T^\ast) \bH)\big), v_N)_{L^2(\Gamma)} \quad \text{for all}~~v_N \in L^2(\Gamma),
\end{split} \]
hence we have the characterization as in \eqref{charlambda}.
\end{proof}
\ \\[1ex]
 From  \eqref{charlambda} it follows that if $\bu_T^\ast$ has smoothness $\bu_T^\ast \in H^m(\Gamma)^3$ and the manifold is sufficiently smooth (hence $\bH$ sufficiently smooth) then we have $\lambda \in H^{m-1}(\Gamma)$. Note that if $\bH=0$ and $g_N=0$ then $\lambda=0$.

\begin{remark}\label{rem_aug} \rm In the proof above we used that the form $a(\cdot,\cdot)$ is elliptic on $\bV_T$, the subspace of $\bV_\ast$
consisting of all functions that satisfy the second equation in \eqref{weak1a}.
Note the inequality
\begin{align*}
a(\bu,\bu) &\ge \varepsilon\|E_s(\bu_T^\ast)\|^2_{L^2(\Gamma)}-\frac{\varepsilon}{1-\varepsilon}\|\bH u_N\|^2_{L^2(\Gamma)}+\|\bu\|^2_{L^2(\Gamma)}\\
&\ge  \varepsilon\|E_s(\bu_T^\ast)\|^2_{L^2(\Gamma)}+\big(1-\frac{\varepsilon}{1-\varepsilon} \|\bH\|_{L^\infty(\Gamma)}^2\big)\|\bu\|^2_{L^2(\Gamma)}~~\forall\,\bu\in\bV_\ast,~~\forall~\varepsilon < 1.
\end{align*}
With $\varepsilon_0:=\frac12(1+\|\bH\|_{L^\infty(\Gamma)}^2)^{-1}$ and the Korn inequality \eqref{korn} we get
\begin{align*}
 a(\bu,\bu) \geq \varepsilon_0\big(\|E_s(\bu_T^\ast)\|^2_{L^2(\Gamma)} + \|\bu\|^2_{L^2(\Gamma)}\big)\geq \frac12 \varepsilon_0 c_K^2 \|\bu_T\|_1^2 + \varepsilon_0 \|u_N\|_{L^2(\Gamma)}^2
\end{align*}
 for all $\bu \in \bV_\ast$. Hence,  the bilinear form $a(\cdot,\cdot)$ is also elliptic on $\bV_\ast$.
Note that the ellipticity constant depends on the curvature of $\Gamma$.
%The bilinear form $a(\cdot,\cdot)$ is not necessarily elliptic on $\bV_\ast$.
%^One can use an augmented Lagrangian formulation in which, for a given fixed $\eta \geq 0$,   a modified bilinear form
%A finite element method for \eqref{weak1a} may lead to a system of linear algebraic equation as in \eqref{SLAE}
%with the singular submatrix $A$. A standard technique to improve numerical properties is to consider the augmented Lagrangian
%formulation. In this approach, one amends the bilinear form by the least squares term corresponding to the constraint.
%Therefore we introduce
%\begin{equation} \label{defaeta}
% a_\eta(\bu,\bv):= a(\bu,\bv)+ \eta (\bu\cdot\bn,\bv\cdot \bn)_{L^2(\Gamma)}
%\end{equation}
%is introduced. The augmentation parameter $\eta$ can be chosen sufficiently large such that
%\begin{equation} \label{ell1}
%a_\eta(\bu,\bu) \geq c \|\bu\|_{V_\ast}^2 \quad \text{for all}~~\bu \in \bV_\ast,
%\end{equation}
%holds, with an ellipticity constant $c>0$. We do not treat this variant, because in the analysis below we do not need such a stabilization with an augmented Lagrangian and numerical experiments have shown that both with respect to discretization accuracy and conditioning of the linear systems there is no significant improvement if one uses $\eta >0$ instead of $\eta=0$.\\
%\textbf{\small MO: Note that we do add some augmentation implicitly with $\eta=1$ due to the zero order term, i.e we use $\bu\cdot\bv$ in our weak formulation instead of $\bu_T\cdot \bv_T$.}
\end{remark}
\ \\
%Hence,
%ddddd
%\begin{equation} \label{est9}
%  \|\bu\|_{L^2(\Gamma)}^2 \lesssim \|\bu\|_{V_\ast}^2 \lesssim \tilde a(\bu,\bu) \lesssim \|\bu\|_{V_\ast}^2
% \lesssim \|\bu\|_{1}^2,  \quad \bu \in V_\ast,
%\end{equation}
%where for the last inequality we need the additional smoothness $\bu \in V$.
%
%The augmented formulation reads: Find $(\bu,\lambda) \in V_\ast \times L^2(\Gamma)$ such that
% \begin{equation} \label{weak1} \begin{split}
%           \tilde a(\bu,\bv) +(\bv\cdot \bn,\lambda)_{L^2(\Gamma)} &=(\mathbf{f},\bv)_{L^2(\Gamma)} \quad \text{for all}~~\bv \in V_\ast, \\
%           (\bu\cdot\bn,\mu)_{L^2(\Gamma)} & = 0 \qquad \text{for all}~~\mu \in L^2(\Gamma).
%                                      \end{split}
% \end{equation}
%Similar to \eqref{weak1a}, the weak formulation \eqref{weak1} is well-posed and has a unique solution $(\tilde\bu,\lambda) \in V_\ast \times L^2(\Gamma)$, which has the property $\tilde \bu \cdot \bn=0$, $\tilde \bu_T= \bu_T^\ast$,    with $\bu_T^\ast$ the unique solution of \eqref{vectLaplace}.
%Hence, the formulation \eqref{weak1} is \emph{consistent} to the original one, and we shall apply a finite element method to \eqref{weak1}.
\begin{remark} \rm
 Instead of the weak formulation in \eqref{weak1a} one can also consider a penalty formulation, without using a Lagrange multiplier $\lambda$. Such an approach is used for a similar Bochner-Laplace problem in \cite{hansbo2016analysis}. This formulation is as follows: determine $\bu \in \bV_\ast$ such that
\begin{equation}
   a(\bu,\bv)+ \eta (\bu\cdot\bn,\bv\cdot \bn)_{L^2(\Gamma)}= (\mathbf{f},\bv)_{L^2(\Gamma)} \quad \text{for all}~~\bv \in \bV_\ast,
\end{equation}
with $\eta >0$ (sufficiently large).
From ellipticity and continuity it follows that this weak formulation has a unique solution, denoted by $\hat \bu$. Opposite to the solution $\bu^\ast$ of \eqref{weak1a}, the solution $\hat \bu$ does not have the property $\hat \bu \cdot \bn=0$, and in general $\hat \bu_T \neq \bu_T^\ast$ holds. Using standard arguments one easily derives the error bound
\[
 \|\hat \bu_T- \bu_T^\ast\|_{V^\ast} \leq c \eta^{-\frac12} \|\mathbf{f}\|_{L^2(\Gamma)}.
\]
Hence, as usual in this type of  penalty method, one has to take $\eta$ sufficiently large depending on the desired accuracy of the approximation.
\end{remark}
\ \\
\begin{remark} \rm
The  analysis of well-posedness above and the finite element method presented in the next section have immediate extensions to the case of the Bochner Laplacian on $\Gamma$. For this, one replaces the bilinear form in \eqref{vectLaplace} by
\[
a^B(\bu,\bv) :=  \int_\Gamma (\gradG \bu:\gradG \bu + \bu \cdot \bv) \, ds, \qquad \bu,\bv \in \bV,
\]
and instead of \eqref{idfund} one uses $\gradG \bu=\gradG \bu_T+u_N\bH$ for further analysis. In this case,  Korn's inequality \eqref{korn}  is replaced by Poincare's inequality on $\Gamma$ (cf.~\cite{hansbo2016analysis}). Based on the second equality in \eqref{LaplacianALL} and the tangential variational formulation for the Bochner Laplacian, one can also consider the bilinear form
\[
a^H(\bu,\bv) :=  \int_\Gamma (\gradG \bu:\gradG \bu + (1+K)\bu \cdot \bv) \, ds, \qquad \bu,\bv \in \bV,
\]
for an equation with the Hodge Laplacian.
This formulation, however, is less convenient for the analysis of  well-posedness in the framework of this paper, since the Gauss curvature $K$ in general does not have a fixed sign. Moreover,  in a numerical method one then has to approximate the Gauss curvature $K$ based on a ``discrete'' (e.g., piecewise planar) surface  approximation, which is known to be a delicate numerical issue.
\end{remark}

\section{Trace Finite Element Method}\label{s_TraceFEM}
For the discretization of the variational problem \eqref{weak1a} we use the trace finite element approach (TraceFEM)~ \cite{ORG09}. For this, we assume a fixed polygonal domain $\Omega \subset \R^3$ that strictly contains $\Gamma$. We use a family of shape regular tetrahedral triangulations $\{\T_h\}_{h >0}$ of $\Omega$. The subset of tetrahedra that have a nonzero intersection with $\Gamma$ is collected in the set denoted by $\T_h^\Gamma$. For simplicity, in the analysis of the method,  we assume
$\{\T_h^\Gamma\}_{h >0}$ to be quasi-uniform. The domain formed by all tetrahedra in $\T_h^\Gamma$ is denoted by $\OGamma:=\text{int}(\overline{\cup_{T \in \T_h^\Gamma} T})$.
On $\T_h^\Gamma$ we use a standard finite element space of continuous functions that are piecewise polynomial of degree $k$. This so-called \emph{outer finite element space} is denoted by $V_h^k$.  In the stabilization terms added to the finite element formulation (see below), we need an extension of the normal vector field $\bn$ from $\Gamma$ to $\OGamma$. For this we use $\bn^e = \nabla d$, where $d$ is the signed distance function to $\Gamma$. In practice, this signed distance function is often not available and we then use approximations as discussed in Remark~\ref{remimplementation}. Another aspect related to implementation is that in practice it is often  not easy to compute integrals over the surface $\Gamma$ with high order accuracy. This may be due to the fact that $\Gamma$ is defined implicitly as the zero level of a level set function and a parametrization of $\Gamma$ is not available. This issue of ``geometric errors'' and of a feasible approximation $\Gamma_h \approx \Gamma$ will also be
addressed in Remark~\ref{remimplementation}. Below in the presentation and analysis of the TraceFEM we use the exact extended normal $\bn=\bn^e$ and we assume exact integration over $\Gamma$.

We introduce the stabilized bilinear forms, with $\bU:=\{\, \bv \in H^1(\OGamma)^3~|~\bv_{|\Gamma} \in \bV\,\}$, $M:=H^1(\OGamma)$,
\begin{align*}
 A_h(\bu,\bv) &:= a(\bu,\bv) + \rho \int_{\OGamma} (\nabla \bu \bn)\cdot (\nabla  \bv \bn) \, dx, \qquad \bu,\bv \in \bU, \\
b(\bu,\mu)& := (\bu\cdot \bn,\mu)_{L^2(\Gamma)} + \tilde \rho  \int_{\OGamma} (\bn^T \nabla \bu \bn) (\bn \cdot \nabla \mu) \, dx \\
  &= (u_N,\mu)_{L^2(\Gamma)} + \tilde \rho  \int_{\OGamma} (\bn\cdot \nabla u_ N)(\bn \cdot \nabla \mu) \, dx,     \qquad \bu \in \bU, ~ \mu \in M.
\end{align*}
  Such ``volume normal derivative'' stabilizations have recently been studied in \cite{grande2017higher,burmanembedded}.
The parameters in the stabilizations may be $h$-dependent, $\rho \sim h^{m},~\tilde \rho \sim h^{\tilde m}$. One can consider different scalings, i.e., $m \neq \tilde m$. From the analysis below it follows that the best choice is $m=\tilde m$. To simplify the presentation we set $\rho=\tilde \rho$. Based on the analysis in \cite{grande2017higher} of \emph{scalar} surface problems we restrict to
\begin{equation} \label{scalingrho}
 h \lesssim \rho=\tilde \rho \lesssim h^{-1}.
\end{equation}
Here and further in the paper we write $x\lesssim y$ to state that the inequality  $x\le c y$ holds for quantities $x,y$ with a constant $c$, which is independent of the mesh parameter $h$ and the position of $\Gamma$ over the background mesh. Similar we give sense to $x\gtrsim y$; and $x\sim y$ will mean that both $x\lesssim y$ and $x\gtrsim y$ hold.
For fixed $k,l \geq 1$ we take finite element spaces
\[ \bU_h:= (V_h^k)^3\subset \bU, \quad M_h :=V_h^l\subset M,
\]
  for the velocity $\bu$ and  the Lagrange multiplier $\lambda$, respectively. The finite element method (TraceFEM) that we consider is as follows: determine $(\bu_h, \lambda_h) \in \bU_h \times M_h$ such that
\begin{equation} \label{discrete}
 \begin{aligned}
  A_h(\bu_h,\bv_h) + b(\bv_h,\lambda_h) & =(\mathbf{g},\bv_h)_{L^2(\Gamma)} &\quad &\text{for all } \bv_h \in \bU_h \\
  b(\bu_h,\mu_h) & = 0 &\quad &\text{for all }\mu_h \in M_h.
 \end{aligned}
\end{equation}

\begin{remark} \label{remimplementation}
 \rm As noted above, in the implementation of this method one typically replaces $\Gamma$ by an approximation $\Gamma_h \approx \Gamma$ such that integrals over $\Gamma_h$ can be efficiently computed. Furthermore, the exact normal $\bn$ is approximated by  $\bn_h \approx \bn $. In the literature on finite element methods for surface PDEs there are standard procedures resulting in a piecewise planar surface approximation $\Gamma_h$ with ${\rm dist}(\Gamma,\Gamma_h) \lesssim h^2$.
 If one is interested in surface FEM with higher order surface approximation, we refer to the recent paper \cite{grande2017higher}, where one finds an efficient method based on an isoparametric mapping derived from a level set representation of $\Gamma$.
 %In the recent paper \cite{grande2017higher} an efficient method (based on an isoparametric mapping derived from a level set representation of $\Gamma$) for contructing a higher order surface approximation $\Gamma_h$ is introduced.
 In \cite{demlow2009higher} another higher order surface approximation method is treated. In the numerical experiments in section~\ref{sectExp} we use a piecewise planar surface approximation. Also for the  construction of suitable normal approximations $\bn_h \approx \bn$ several techniques are available in the literature. One possibility is to use $\bn_h(\bx)=\frac{\nabla \phi_h(\bx)}{\|\nabla \phi_h(\bx)\|_2}$, where $\phi_h$ is a finite element approximation of a level set function $\phi$ which characterizes $\Gamma$. This technique is used in section~\ref{sectExp}. In this paper we do not analyze the effect of such geometric errors, i.e., we only analyze the finite element method \eqref{discrete}.
\end{remark}

\section{Error analysis of TraceFEM}\label{s_error}
In this section we present an error analysis of the TraceFEM \eqref{discrete}.
We first address consistency of this stabilized formulation. The solution  $(\bu^\ast, \lambda)$ of \eqref{weak1a}, which is defined only on $\Gamma$, can be extended by constant values along normals to  a neighborhood $\mathcal{O}(\Gamma)$ of $\Gamma$ such that $\OGamma\subset\mathcal{O}(\Gamma)$. This extended solution $((\bu^\ast)^e, \lambda^e)$ is also denoted by $(\bu^\ast, \lambda)$. Hence we have $ \nabla \bu^\ast \bn=0$,  $\bn \cdot \nabla (\bu^\ast \cdot \bn)=0$, $\bn\cdot \nabla \lambda =0$ on $\OGamma$. Using these properties and $(\bU_h)_{|\Gamma} \subset \bV_\ast$, $(M_h)_{|\Gamma} \subset L^2(\Gamma)$ we get the following \emph{consistency result}:
\begin{equation} \label{consisdiscrete}
 \begin{aligned}
  A_h(\bu^\ast,\bv_h) + b(\bv_h,\lambda) = a(\bu^\ast,\bv_h)+ (\bv_h \cdot \bn,\lambda)_{L^2(\Gamma)}& =(\mathbf{g},\bv_h)_{L^2(\Gamma)}&\quad&\forall~\bv_h \in \bU_h, \\
  b(\bu^\ast,\mu_h)= (\bu^\ast \cdot \bn,\mu_h)_{L^2(\Gamma)} & = 0 &\quad&\forall~\mu_h \in M_h.
 \end{aligned}
\end{equation}
We now address continuity of the bilinear forms. For this we introduce the semi-norms
\begin{align}
 \|\bu\|_U^2 & :=A_h(\bu,\bu), \quad \bu \in \bU, \label{defUh} \\
  \|\mu\|_M^2 &:= \|\mu\|_{L^2(\Gamma)}^2 + \rho \|\bn \cdot \nabla \mu\|_{L^2(\OGamma)}^2, \quad \mu \in M. \label{defmuast}
\end{align}
\begin{lemma} \label{lemcont} The following holds
\begin{align}
 A_h(\bu,\bv) & \leq \|\bu\|_U\|\bv\|_U \quad \text{for all }\bu,\bv \in \bU, \label{cont1}\\
 b(\bu,\mu) & \leq \|\bu\|_U \|\mu\|_M \quad \text{for all }\bu \in \bU, \mu \in M. \label{cont2}
\end{align}
\end{lemma}
\begin{proof}
 The result in \eqref{cont1} follows from Cauchy-Schwarz inequalities. Note that  due to $\nabla \bn \bn =0$ and the symmetry of $\nabla \bn$ we obtain $\bn\cdot \nabla u_N=\bn \cdot\nabla( \bu \cdot \bn)= \bn^T \nabla \bu \bn + \bn^T \nabla \bn \bu =  \bn^T\nabla \bu \bn  $. Hence, $|\bn \cdot\nabla  u_N| \leq \|\nabla \bu \bn\|_2$ holds (pointwise at $x \in \mathcal{O}(\Gamma)$).  Using this we get for $\bu \in \bU$, $\mu \in M$:
\begin{align}
 |b(\bu,\mu)| & \leq |(\bu \cdot \bn, \mu)_{L^2(\Gamma)}| +\rho |(\bn \cdot \nabla u_N  , \bn \cdot \nabla \mu)_{L^2(\OGamma)}| \nonumber \\
 & \leq \|\bu\|_{L^2(\Gamma)}\|\mu\|_{L^2(\Gamma)}+ \rho \|\bn \cdot \nabla u_N\|_{L^2(\OGamma)}\|\bn \cdot \nabla \mu\|_{L^2(\OGamma)} \nonumber \\
 & \leq \big(\|\bu\|_{L^2(\Gamma)}^2 + \rho  \| \nabla \bu\bn\|_{L^2(\OGamma)}^2 \big)^\frac12 \big(\|\mu\|_{L^2(\Gamma)}^2 +\rho \|\bn \cdot \nabla \mu\|_{L^2(\OGamma)}^2 \big)^\frac12 \label{bineq} \\
  & \leq A_h(\bu,\bu)^\frac12   \|\mu\|_M= \|\bu\|_U\|\mu\|_M. \nonumber
\end{align}
This completes the proof.
\end{proof}
\ \\[2ex]
The following result is crucial for the stability and error analysis of the method.
\begin{lemma} \label{lemcrucial} The following uniform norm equivalence holds:
 \begin{equation} \label{fund1}
  h \|v_h\|_{L^2(\Gamma)}^2 + h^2 \|\bn \cdot \nabla v_h\|_{L^2(\OGamma)}^2 \sim \|v_h\|_{L^2(\OGamma)}^2\quad \text{for all}~~v_h \in V_h^k.
\end{equation}
\end{lemma}
\begin{proof}
A fundamental result derived in \cite{grande2017higher}  (Lemma 7.6) is:
\begin{equation} \label{fund1A}
 \|v_h\|_{L^2(\OGamma)}^2 \lesssim h \|v_h\|_{L^2(\Gamma)}^2 + h^2 \|\bn \cdot \nabla v_h\|_{L^2(\OGamma)}^2 \quad \text{for all}~~v_h \in V_h^k.
\end{equation}
(This follows by taking $\Psi={\rm id}$ in the analysis in section 7.2 in \cite{grande2017higher}).
We combine this with the following estimate, cf.~\cite{Hansbo02}:
\begin{equation} \label{fund1B}
 h \|v\|_{L^2(\Gamma)}^2 \lesssim  \|v\|_{L^2(\OGamma)}^2+h^2\|v\|_{H^1(\OGamma)}^2 \quad \text{for all}~v \in H^1(\OGamma),
\end{equation}
and a standard finite element inverse inequality $\|v_h\|_{H^1(\OGamma)} \lesssim h^{-1}\|v_h\|_{L^2(\OGamma)}$ for all $v_h \in V_h^k$.
This completes the proof.
\end{proof}
\ \\[2ex]
Using \eqref{fund1} and \eqref{scalingrho} we get
\begin{equation}\label{elliptic} \begin{split}
 A_h(\bu_h,\bu_h) & = a(\bu_h,\bu_h) + \rho \int_{\OGamma} ( \nabla \bu_h\bn)\cdot ( \nabla  \bu_h\bn ) \, dx \\
  & \gtrsim \|\bu_h\|_{L^2(\Gamma)}^2 +  h \int_{\OGamma} (\nabla \bu_h\bn)\cdot ( \nabla  \bu_h\bn) \, dx \\
 & =\sum_{i=1}^3 \Big( \|(u_h)_i\|_{L^2(\Gamma)}^2 + h \|\bn \cdot \nabla (u_h)_i\|_{L^2(\OGamma)}^2\Big) \\ &  \gtrsim h^{-1} \sum_{i=1}^3 \|(u_h)_i\|_{L^2(\OGamma)}^2 = c h^{-1}\|\bu_h\|_{L^2(\OGamma)}^2,
\end{split}
\end{equation}
for all $\bu_h \in \bU_h$. This  implies that $A_h(\cdot,\cdot)$ \emph{is a scalar product on} $\bU_h$.
Using \eqref{fund1} and $\rho \gtrsim h$ we get
\begin{equation}\label{lower2}
 \|\mu_h\|_M^2  \gtrsim h^{-1} \|\mu_h\|_{L^2(\OGamma)}^2 \quad \text{for all}~~\mu_h \in M_h.
\end{equation}
This in particular implies that $\|\cdot\|_M$ corresponds to a \emph{scalar product on} $M_h$.
We now turn to the discrete inf-sup property.
\begin{lemma} \label{lemdiscreteinfsup}
Take $m \geq 1$.
%For $h$ sufficiently small and $\rho\le c_\rho h^{-1}$, with $c_\rho=O(1)$ sufficiently small,
There exist constants $d_1>0$, $d_2 >0$, independent of $h$ and of how $\Gamma$ intersects the outer triangulation, such that:
 \begin{equation} \label{infsupest}
  \sup_{\bv_h \in (V_h^m)^3} \frac{b(\bv_h,\mu_h)}{\|\bv_h\|_U} \geq d_1(1- d_2\sqrt{\rho h}) \|\mu_h\|_M \quad \text{for all}~~\mu_h \in V_h^m.
 \end{equation}
\end{lemma}
\begin{proof} Take $\mu_h \in V_h^m$.
 Note that
 \[
  b(\mu_h \bn, \mu_h)= \|\mu_h\|_M^2.
 \]
Take $\bv_h:= I_m (\mu_h \bn) \in (V_h^m)^3$, where $I_m$ is the nodal (Lagrange) interpolation operator. The latter is well defined, because both $\mu_h$ and $\bn$ are continuous in $\OGamma$. Now note, cf.~\eqref{bineq},
\begin{equation} \label{eq1}
\begin{split}
& |b(\bv_h, \mu_h)|   \geq |b(\mu_h \bn, \mu_h)|- |b(I_m (\mu_h\bn) -\mu_h\bn,\mu_h)| \\  &=\|\mu_h\|_M^2 - |b(I_m (\mu_h\bn) -\mu_h\bn,\mu_h)| \\
 & \geq \big(\|\mu_h\|_M- \big(\|I_m (\mu_h\bn) -\mu_h\bn\|_{L^2(\Gamma)}^2 + \rho  \| \nabla (I_m (\mu_h\bn) -\mu_h\bn) \bn\|_{L^2(\OGamma)}^2 \big)^\frac12 \big)\|\mu_h\|_M.
\end{split} \end{equation}
From $E_s(\mu_h \bn)=\mu_h \bH$ we get $a(\mu_h \bn,\mu_h \bn) \lesssim \|\mu_h\|_{L^2(\Gamma)}^2$ and using this and $\bH \bn=0$ we obtain
\begin{equation} \label{eq2}
 \|\bv_h\|_U \leq \|\mu_h \bn\|_U+ \|I_m (\mu_h\bn) -\mu_h\bn\|_U  \lesssim \|\mu_h\|_M + \|I_m (\mu_h\bn) -\mu_h\bn\|_U.
\end{equation}
We now consider the terms with $I_m (\mu_h\bn) -\mu_h\bn$ in \eqref{eq1} and \eqref{eq2}. We use
standard element-wise interpolation bounds for the Lagrange interpolant, the identity $|\mu_h|_{H^{m+1}(K)}=0$ for the $H^{m+1}$ seminorm of $\mu_h$ over any tetrahedron $K \in \cT_h^\Gamma$, the inverse inequality $\|\mu_h\|_{H^m(K)}\le c h^{-m}\|\mu_h\|_{L^2(K)}$, \eqref{fund1} and the local variant of the estimate \eqref{fund1B}.
%the bounds in \eqref{interror},  $\rho \lesssim h^{-1}$ and the following Hansbo/inverse inequality for $\mu_h \in M_h$:
%\[
%  \|\mu_h\|_1^2 \lesssim h^{-1}\|\mu_h\|_{H^1(\OGamma)}^2 \lesssim h^{-3}\|\mu_h\|_{L^2(\OGamma)}^2.
%\]
We then obtain
\begin{align*}
 & \|I_m (\mu_h\bn) -\mu_h\bn\|_U^2  \\ & =  a(I_m (\mu_h\bn) -\mu_h\bn,I_m (\mu_h\bn) -\mu_h\bn) +\rho \|\nabla\big(I_m (\mu_h\bn) -\mu_h\bn\big)\bn\|_{L^2(\OGamma)}^2 \\
 & \lesssim \sum_{K\in\cT_h^\Gamma}\|I_m (\mu_h\bn) -\mu_h\bn\|_{H^1(K\cap \Gamma)}^2 + \rho \|\nabla(I_m (\mu_h\bn) -\mu_h\bn)\bn\|_{L^2(\OGamma)}^2 \\
 & \lesssim \sum_{K\in\cT_h^\Gamma}\left\{(h^{-1}+\rho)\|I_m (\mu_h\bn) -\mu_h\bn\|_{H^1(K)}^2+h|I_m (\mu_h\bn) -\mu_h\bn|_{H^2(K)}^2\right\}\\
 & \lesssim (h^{-1}+\rho)h^{2m}\sum_{K\in\cT_h^\Gamma}|\mu_h\bn|_{H^{m+1}(K)}^2 \lesssim (h^{-1}+\rho)h^{2m}\sum_{K\in\cT_h^\Gamma}\|\mu_h\|_{H^{m}(K)}^2\\
 & \lesssim (h^{-1}+\rho)\sum_{K\in\cT_h^\Gamma}\|\mu_h\|_{L^2(K)}^2 \sim (1+\rho h)h^{-1}\|\mu_h\|_{L^2(\OGamma)}^2\\
 &\lesssim(1+\rho h)(\|\mu_h\|_{L^2(\Gamma)}^2+h\|\bn\cdot\nabla\mu_h\|_{L^2(\OGamma)}^2)
% \lesssim(1+\rho h)(|\mu_h|_{L^2(\Gamma)}^2+h|\bn\cdot\mu_h|_{L^2(\Gamma)}^2)\\
  \lesssim(1+\rho h)\|\mu_h\|_M^2 \lesssim \|\mu_h\|_M^2.
\end{align*}
From this and \eqref{eq2} we get
\begin{equation} \label{aux1}
 \|\bv_h\|_U\lesssim \|\mu_h\|_M.
\end{equation}
With similar arguments we bound the interpolation terms in \eqref{eq1}:
\begin{equation}\label{aux}
\begin{aligned}
 &  \|I_m (\mu_h\bn) -\mu_h\bn\|_{L^2(\Gamma)}^2 + \rho  \|\nabla \big(I_m (\mu_h\bn) -\mu_h\bn\big)\bn\|_{L^2(\OGamma)}^2 \\
 & \lesssim \sum_{K\in\cT_h^\Gamma}\left\{h^{-1}\|I_m (\mu_h\bn) -\mu_h\bn\|_{L^2(K)}^2+h\|I_m (\mu_h\bn) -\mu_h\bn\|_{H^1(K)}^2\right\} \\ & \quad + \rho \|\nabla \big(I_m (\mu_h\bn) -\mu_h\bn\big)\bn\|_{L^2(\OGamma)}^2 \\
 & \lesssim \sum_{K\in\cT_h^\Gamma}\left\{h^{-1}\|I_m (\mu_h\bn) -\mu_h\bn\|_{L^2(K)}^2+(h+\rho)\|I_m (\mu_h\bn) -\mu_h\bn\|_{H^1(K)}^2\right\} \\
 &\lesssim(h+\rho)h^{2m} \sum_{K\in\cT_h^\Gamma}|\mu_h\bn|_{H^{m+1}(K)}^2\\
 &\lesssim(h^2+\rho h)\|\mu_h\|_M^2 \lesssim \rho h \|\mu_h\|_M^2
\end{aligned}
\end{equation}
Combining this with the results in \eqref{eq1} and \eqref{aux1} completes the proof.
\end{proof}
\ \\[1ex]
\begin{corollary} \label{corolstab} Take $m\geq 1$.
Consider $\rho=c_\alpha h^{1-\alpha}$, $\alpha \in[0,2]$, and assume $h \leq h_0 \leq 1$. Take $c_\alpha$ such that $0< c_\alpha < d_2^{-2} h_0^{\alpha-2}$ with $d_2$ as in \eqref{infsupest}. Then there exists a constant $d >0$, independent of $h$ and of how $\Gamma$ intersects the outer triangulation, such that:
 \begin{equation} \label{infsupestA}
  \sup_{\bv_h \in (V_h^m)^3} \frac{b(\bv_h,\mu_h)}{\|\bv_h\|_U} \geq d \|\mu_h\|_M \quad \text{for all}~~\mu_h \in V_h^m.
 \end{equation}
\end{corollary}

\begin{assumption} \label{ass1}
 In the remainder we restrict to $\rho=c_\alpha h^{1-\alpha}$, $\alpha\in [0,2]$,  with $c_\alpha$ as in Corollary~\ref{corolstab}.
\end{assumption}

\begin{corollary} \label{corinfsup}
 For $b(\cdot,\cdot)$ the  discrete inf-sup property holds for the pair of spaces $(\bU_h,M_h)=\big( (V_h^k)^3, V_h^l\big)$ with $ 1 \leq l \leq k$.
The constant in the discrete inf-sup estimate can be taken independent of  $h$  and of how $\Gamma$ intersects the outer triangulation, but depends on $k$.
\end{corollary}
\ \\[1ex]
From the fact that $A_h(\cdot,\cdot)$ defines a scalar product on $\bU_h$ and the discrete inf-sup property of $b(\cdot,\cdot)$ on $\bU_h \times M_h$ it follows that the discrete problem \eqref{discrete} has a unique solution $(\bu_h,\lambda_h)$.

For the remainder of the error analysis we  apply standard theory of saddle point problems. We introduce the bilinear form
\begin{equation} \label{defk}
 \bA_h\big((\bu,\lambda),(\bv,\mu)\big):= A_h(\bu,\bv)+ b(\bv,\lambda)+b(\bu,\mu), \quad (\bu,\lambda), (\bv,\mu) \in \bU \times M.
\end{equation}
\begin{theorem} \label{thm2}  Let $(\bu^\ast,\lambda)\in \bV_\ast \times L^2(\Gamma)$ be the solution of \eqref{weak1a} and assume that this solution is sufficiently smooth. Furthermore, let  $(\bu_h,\lambda_h) \in (V_h^k)^3\times V_h^l$ be the solution  of \eqref{discrete}. For $1 \leq l \leq k$ the following discretization error bound holds:
\begin{equation}\label{discrbound}
\begin{split}
 & \|\bu^\ast -\bu_h\|_U + \|\lambda - \lambda_h \|_M  \\ & \lesssim h^k\big( 1+ (\rho h)^\frac12\big)\|\bu^\ast\|_{H^{k+1}(\Gamma)} + h^{l+1} \big(1+(\rho/h)^\frac12\big) \|\lambda\|_{H^{l+1}(\Gamma)}.
 \end{split}
\end{equation}
\end{theorem}
\begin{proof}
Using the consistency property \eqref{consisdiscrete}, the continuity results derived in Lemma~\ref{lemcont},  ellipticity of $A_h(\cdot,\cdot)$ on $\bU_h$ and the discrete inf-sup property for $b(\cdot,\cdot)$ we obtain, for arbitrary $(\bw_h,\xi_h) \in \bU_h \times M_h$:
\begin{align*}
 & \big(\|\bu_h- \bw_h\|_U^2 +\|\lambda_h- \xi_h\|_M^2\big)^\frac12 \lesssim \sup_{(\bv_h,\mu_h)\in \bU_h \times M_h} \frac{\bA_h\big( (\bu_h- \bw_h,\lambda_h- \xi_h),(\bv_h,\mu_h)\big)}{(\|\bv_h\|_U^2+\|\mu_h\|_M^2)^\frac12} \\ & = \sup_{(\bv_h,\mu_h)\in \bU_h \times M_h} \frac{ \bA_h\big( (\bu^\ast- \bw_h,\lambda- \xi_h),(\bv_h,\mu_h)\big)}{(\|\bv_h\|_U^2+\|\mu_h\|_M^2)^\frac12} \lesssim \|\bu^\ast - \bw_h\|_U + \|\lambda - \xi_h\|_M.
\end{align*}
Hence, with a triangle inequality we get the Cea-type discretization error bound
\begin{equation} \label{discrerror}
\|\bu^\ast -\bu_h\|_U + \|\lambda - \lambda_h \|_M \lesssim
 \inf_{(\bv_h,\mu_h)\in \bU_h \times M_h}\big(\|\bu^\ast -\bv_h\|_U + \|\lambda - \mu_h \|_M\big).
\end{equation}
For  $(\bv_h,\mu_h)\in \bU_h \times M_h$ we take the interpolants $\bv_h=I_k\big((\bu^\ast)^e\big)$, $\mu_h=I_l(\lambda^e)$,   and assume sufficient smoothness of $\bu^\ast$ and hence of $\lambda$, cf. \eqref{charlambda}. Then, thanks to the interpolation properties of polynomials and their traces, cf., e.g., \cite{reusken2015analysis}, we have the estimates:
\begin{align*}
 & \|\bu^\ast -\bv_h\|_U + \|\lambda - \mu_h \|_M \\ & \lesssim \|\bu^\ast - \bv_h\|_1+ \rho^\frac12 \|\bu^\ast -\bv_h\|_{H^1(\OGamma)} + \|\lambda - \mu_h\|_{L^2(\Gamma)} + \rho^\frac12 \|\lambda - \mu_h\|_{H^1(\OGamma)} \\
  &\lesssim h^k \|\bu^\ast\|_{H^{k+1}(\Gamma)}+ \rho^\frac12 h^k \|\bu^\ast\|_{H^{k+1}(\OGamma)} +h^{l+1} \|\lambda\|_{H^{l+1}(\Gamma)} + \rho^\frac12 h^l\|\lambda\|_{H^{l+1}(\OGamma)}\\
   & \lesssim h^k\big(1 +(\rho h)^\frac12\big)\|\bu^\ast\|_{H^{k+1}(\Gamma)}+ h^{l+1}\big(1+ (\rho/h)^\frac12 \big)\|\lambda\|_{H^{l+1}(\Gamma)},
\end{align*}
which in combination with \eqref{discrerror} yields the desired result.
\end{proof}
\ \\[1ex]
\begin{corollary} \label{cor1} \rm Assume that the solution of \eqref{weak1a} is sufficiently smooth. We obtain  for $l=k \geq 1$  the optimal error bound:
\begin{equation} \label{dis1}
 \|\bu^\ast -\bu_h\|_U + \|\lambda - \lambda_h \|_M \lesssim h^k (\|\bu^\ast\|_{H^{k+1}(\Gamma)} +\|\lambda\|_{H^{k+1}(\Gamma)}).
\end{equation}
For $l=k-1 \geq 1$ and with $\rho \sim h$ we obtain the optimal error bound:
\begin{equation} \label{dis2}
 \|\bu^\ast -\bu_h\|_U + \|\lambda - \lambda_h \|_M \lesssim h^k (\|\bu^\ast\|_{H^{k+1}(\Gamma)} +\|\lambda\|_{H^{k}(\Gamma)}).
\end{equation}
If $\lambda=0$, cf. \eqref{charlambda}, the bound \eqref{dis2} holds for $l=k-1 \geq 1$ and for \emph{any} $\rho$ that fulfills Assumption~\ref{ass1}.   Using \eqref{charlambda}  we can bound the norm of $\lambda$ in terms of the normal component of the data $g_N$ and $\bu^\ast$:
\begin{equation} \label{replace}
  \|\lambda\|_{H^{m}(\Gamma)} \lesssim \|g_N\|_{H^{m}(\Gamma)}+ \|\bu^\ast\|_{H^{m+1}(\Gamma)}.
\end{equation}
Note that for the original problem \eqref{vectLaplace} the data $\mathbf{f}= \bg$ satisfy $g_N=0$.
\end{corollary}

\section{$L^2$-error bound} \label{s_L2}
In this section we use a standard duality argument to derive an optimal $L^2$-norm discretization error bound, based on a  regularity assumption for the problem \eqref{vectLaplace}. We note that in the analysis we need the assumption $\rho \sim h$.

\begin{theorem} \label{thmL2}
Assume that \eqref{vectLaplace} satisfies the regularity estimate
\begin{equation} \label{regu}
 \|\bu_T\|_{H^2(\Gamma)} \lesssim \|\bbf\|_{L^2(\Gamma)} \quad \text{for all $\bbf \in L^2(\Gamma)^3$ with $\bbf\cdot\bn = 0$}.
\end{equation}
%Let the assumptions as in Theorem~\ref{thm2} be satisfied and
Take $\rho \sim h$, $l=k \geq 1$ or $l=k-1 \geq 1$. The following error estimate holds:
\begin{equation} \label{L2error}
 \|\bu^\ast - \bP \bu_h\|_{L^2(\Gamma)} \lesssim h^{k+1}\big( \|\bu^\ast \|_{H^{k+1}(\Gamma)} + \|\lambda\|_{H^{l+1}(\Gamma)}\big).
\end{equation}
\end{theorem}
\begin{proof}
 We consider the problem \eqref{vectLaplace} with $\bbf_e:=\bP(\bu^\ast-\bu_h)= \bu^\ast- \bP \bu_h$. We take $\bg = \bbf_e$ in \eqref{weak1a}, hence $g_N=0$, and the corresponding solution of  \eqref{weak1a} is denoted by $(\bw^\ast,\tau) \in \bV_\ast \times L^2(\Gamma)$. The extensions $(\bw^\ast)^e$, $\tau^e$ are also denoted by $\bw^\ast$ and $\tau$.  From the regularity assumption and $\tau =- \tr(E_s(\bw^\ast)\bH)$ it follows that $(\bw^\ast,\tau) \in H^2(\Gamma)^3 \times H^1(\Gamma)$ and that
 \begin{equation} \label{regu2}
  \|\bw^\ast\|_{H^2(\Gamma)} +\|\tau \|_{H^1(\Gamma)} \lesssim \|\bbf_e\|_{L^2(\Gamma)}
 \end{equation}
holds. With $\bA_h(\cdot,\cdot)$ as in \eqref{defk}   we get the consistency result
\[
\bA_h\big((\bw^\ast,\tau),(\bv,\mu)\big)= (\bbf_e,\bv)_{L^2(\Gamma)} \quad \text{for all}~(\bv,\mu)\in \bU \times M.
\]
We take $(\bv,\mu)=(\bu^\ast-\bu_h,\lambda - \lambda_h) \in \bU \times M$ and using the symmetry of $\bA_h(\cdot,\cdot)$ and the Galerkin orthogonality we obtain
\begin{align*}
 \|\bu^\ast-\bP\bu_h\|_{L^2(\Gamma)}^2 & = \|\bP(\bu^\ast-\bu_h)\|_{L^2(\Gamma)}^2 = (\bP(\bu^\ast-\bu_h),\bu^\ast-\bu_h)_{L^2(\Gamma)} \\ & = (\bg,\bu^\ast-\bu_h)_{L^2(\Gamma)}=    \bA_h\big((\bw^\ast,\tau),(\bu^\ast-\bu_h,\lambda - \lambda_h) \big) \\
 & = k\big((\bu^\ast-\bu_h,\lambda - \lambda_h),(\bw^\ast-\bw_h,\tau - \tau_h)  \big)
\end{align*}
for all $(\bw_h,\tau_h) \in U_h \times M_h$.  We use continuity of $\bA_h(\cdot,\cdot)$ and the results derived in Corollary~\ref{cor1} and thus obtain
\begin{equation}\label{et1}
\|\bu^\ast-\bP\bu_h\|_{L^2(\Gamma)}^2 \lesssim h^k \big(\|\bu^\ast\|_{H^{k+1}(\Gamma)} +\|\lambda\|_{H^{l+1}(\Gamma)}\big)\big( \|\bw^\ast - \bw_h\|_U +\|\tau - \tau_h\|_M\big).
\end{equation}
We take $\bw_h= I_1 (\bw^\ast)$, $\tau_h = I_1 (\tau)$. Using \eqref{regu2} this yields
\begin{align*}
 \|\bw^\ast - \bw_h\|_U & \lesssim \|\bw^\ast - I_1(\bw^\ast)\|_{H^1(\Gamma)} + \rho^\frac12 \|\bw^\ast - I_1(\bw^\ast) \|_{H^1(\OGamma)} \\
  &\lesssim h \|\bw^\ast\|_{H^2(\Gamma)} +  \rho^\frac12 h \|\bw^\ast\|_{H^2(\OGamma)} \\
  & \lesssim h (1 +(\rho h)^\frac12)\|\bw^\ast\|_{H^2(\Gamma)} \lesssim h \|\bbf_e\|_{L^2(\Gamma)} \lesssim h \|\bu^\ast - \bP \bu_h\|_{L^2(\Gamma)}
\end{align*}
and
\begin{align*}
 \|\tau - \tau_h\|_M & \lesssim \|\tau - I_1 (\tau)\|_{L^2(\Gamma)} + \rho^\frac12 \|\tau - I_1 (\tau)\|_{H^1(\OGamma)} \\
  & \lesssim h \|\tau\|_{H^1(\Gamma)} + \rho^\frac12\|\tau\|_{H^1(\OGamma)} \\
   &\lesssim h (1 + (\rho/h)^\frac12)\|\tau\|_{H^1(\Gamma)} \lesssim  h \|\bbf_e\|_{L^2(\Gamma)} \lesssim h \|\bu^\ast - \bP \bu_h\|_{L^2(\Gamma)} ,
\end{align*}
where in the second last inequality we used $\rho \sim h$. Combining these estimates with the result in \eqref{et1} completes the proof.
\end{proof}
\ \\[1ex]
Note that the term $ \|\lambda\|_{H^{l+1}(\Gamma)}$ in \eqref{L2error} can be replaced by $ \|\bu^\ast\|_{H^{l+2}(\Gamma)}$, cf.~\eqref{replace}.
From the proof above it follows that we do not need the assumption $\rho \sim h$ for the special case $\lambda=0$.

We address the question how  accurate the discrete solution $\bu_h$ satisfies the tangential condition $\bu\cdot\bn=0$ on $\Gamma$. Due to the zero order term $\int_\Gamma\bu\cdot\bv\,ds$ in the definition of the bilinear form $a(\cdot,\cdot)$ in \eqref{defaalt}, the norm $\|\cdot\|_U$ in \eqref{discrbound} gives control of the normal components, and hence we get
\[
\|\bu_h\cdot\bn\|_{L^2(\Gamma)}\le\|\bu_h-\bu^\ast\|_{U}.
\]
Therefore, under the assumptions of the Corollary~\ref{cor1} we obtain the estimate
\begin{align} \label{eq:normalError}
\|\bu_h\cdot\bn\|_{L^2(\Gamma)}&\lesssim h^k.
\end{align}
%If the zero order term is not present in the original problem, then the same estimate of $\bu_h\cdot\bn$ can be obtained by introducing the augmentation as discussed in Remark~\ref{rem_aug}.
%For the special case $l=k=1$ and $\rho\sim h$ the estimate \eqref{eq:normalError} can be improved as follows.
%Note that the second equation in \eqref{discrete} provides only a limited control of $\bu_h\cdot\bn$ depending on the polynomial order $l$ of Lagrange multiplier space $M_h$.
Another bound on $\bu_h\cdot\bn$ can be derived from the second equation in \eqref{discrete}.
Denote by $P_l$ the  orthogonal projection onto $M_h$ with respect to the $(\cdot,\cdot)_M$ scalar product, then the second equation in  \eqref{discrete} implies $P_l(\bu_h\cdot\bn)=0$. Therefore, we have
\[
\|\bu_h\cdot\bn\|_{L^2(\Gamma)}\le\|\bu_h\cdot\bn\|_{M}=\|(I-P_l)\bu_h\cdot\bn\|_{M}=\inf_{\mu_h\in M_h}\|\bu_h\cdot\bn-\mu_h\|_{M}.
\]
%To bound the term on the right-hand side, let $\mu_h=I_l(\bu_h\cdot\bn)$ (the Lagrange interpolant is well-defined since
%$\bu_h\cdot\bn$ is continuous on $\OGamma$) and repeat the arguments in \eqref{aux}. We get
%\begin{align} \label{eq:normalErrorImproved}
%\|\bu_h\cdot\bn\|_{L^2(\Gamma)} &\lesssim \rho h\, \|\bu_h\|_{L^2(\Gamma)}
% \lesssim h^2 \qquad\text{for } l=k=1,~\rho\sim h.
%\end{align}
%\textbf{TJ: Using \eqref{aux}, which states
%\begin{equation*}
%\|I_m (\mu_h\bn) -\mu_h\bn\|_{M}^2 \lesssim \rho h \|\mu_h\|_M^2 \Leftrightarrow \|I_m (\mu_h\bn) -\mu_h\bn\|_{M} \lesssim \sqrt{\rho h} \|\mu_h\|_M
%\end{equation*}
%would only result in something like
%\begin{equation*}
%\|\bu_h\cdot\bn\|_{L^2(\Gamma)} \lesssim \sqrt{\rho h}\, \|\bu_h\|_{L^2(\Gamma)}.
%\end{equation*}
%}

\section{Condition number estimate} \label{s_cond}
It is well-known \cite{OlshanskiiReusken08,burmanembedded} that for unfitted finite element methods there is an issue concerning algebraic stability, in the sense that the matrices that represent the discrete problem can have very bad conditioning due to small cuts in the geometry.
Stabilization methods have been developed which remedy this stability problem, see, e.g., \cite{burmanembedded,olshanskii2016trace}. In this section we show that the `volume normal derivative' stabilizations that we use in both bilinear forms $A_h(\cdot,\cdot)$ and $b(\cdot,\cdot)$, with scaling as in \eqref{scalingrho}, remove any possible algebraic instability. More precisely,  we show that the condition number of the stiffness matrix corresponding to the saddle point problem \eqref{discrete} is bounded by $ch^{-2}$, where the constant $c$ is independent of the position of the interface. Furthermore, we present an optimal Schur complement preconditioner.

Let integer $n>0, m>0$ be the number of active degrees of freedom in $\bU_h$ and $M_h$ spaces, i.e., $n= {\rm dim}(\bU_h)$, $m={\rm dim}(M_h)$, and $P_h^U:\,\mathbb{R}^n\to \bU_h$ and $P_h^M:\,\mathbb{R}^m\to M_h$ are canonical mappings between the vectors of nodal values and finite element functions.
%Denote by $M_u$ and $M_\lambda$ the mass matrices for elements in $U_h$ and $M_h$ in $\OGamma$,
%and let
%\[
%\la u,v\ra_M:= \la M_u u,v\ra,\quad \forall~u,v\in\mathbb{R}^n,\qquad
%\la \lambda,\mu\ra_M:= \la M_\lambda\lambda,\mu\ra,\quad \forall~\lambda,\mu\in\mathbb{R}^m,
%\]
Denote by $\la\cdot,\cdot\ra$ and $\|\cdot\|$ the Euclidean scalar product and the norm. For matrices, $\|\cdot\|$ denotes the spectral norm. Now we introduce several matrices. Let
$A, M_u\in\mathbb{R}^{n\times n}$, $B\in\mathbb{R}^{n\times m}$, $M_\lambda, S_M \in \mathbb{R}^{m \times m}$ be such that
\[
\begin{split}
\la A \vec u, \vec v\ra &=  A_h(P_h^U \vec u, P_h^U \vec v),\quad \la B \vec u,\vec \lambda \ra= b(P_h^U \vec u,P_h^M \vec\lambda),\quad
\\
\la M_u \vec u, \vec v\ra&=  (P_h^U \vec u, P_h^U \vec v)_{L^2(\OGamma)},\quad \la M_\lambda\vec \lambda,\vec \mu \ra= (P_h^M \vec \lambda,P_h^M \vec \mu)_{L^2(\OGamma)},\quad \\
\la S_M \vec \lambda , \vec \mu \ra &= (P_h^M \vec \lambda,P_h^M \vec \mu)_{L^2(\Gamma)}+ \rho \big( \nabla (P_h^M \vec \lambda)\bn, \nabla(P_h^M \vec \mu)\bn\big)_{L^2(\OGamma)}
\end{split}
\]
for all $\vec u,\vec v\in\mathbb{R}^n,~~\vec \mu,\,\vec \lambda\in\mathbb{R}^m$. Note that the numerical properties of mass matrices $M_u$ and $M_\lambda$ do not depend on how the surface $\Gamma$  intersects the domain $\OGamma$. Since the family of background meshes is shape regular, these mass matrices have a spectral condition number that is uniformly bounded,  independent of  $h$ and  of how $\Gamma$ intersects the background triangulation $\mathcal{T}_h$. Furthermore, for  the symmetric positive definite matrix $S_M$ we have
\[ \la S_M \vec \lambda, \vec \lambda \ra =\|P_h^M \vec \lambda\|_M^2 \quad \text{for all}~\vec \lambda\in\mathbb{R}^m,
\]
cf.~\eqref{defmuast}.
We also introduce the system matrix and its Schur complement:
\[
\A:=\left[\begin{matrix}
             A & B^T  \\
             B & 0
           \end{matrix}\right],\quad S=B A^{-1} B^T.
% \M:=\left[\begin{matrix}
%
%             M_u & 0  \\
%             0 & M_\lambda
%           \end{matrix}\right].
\]
The algebraic system resulting from the finite element method \eqref{discrete} has the form
\begin{equation}\label{SLAE}
\A \vec x=\vec b,\quad\text{with some}~\vec x, \vec b\in\mathbb{R}^{n+m}.
\end{equation}
We will consider a block-diagonal preconditioner of the matrix $\A$. For this we first analyze preconditioners of the matrices $A$ and $S$. In the following lemma we use spectral inequalities for symmetric matrices.
\begin{lemma} \label{lemprecond} There are strictly positive constants $\nu_{A,1}$, $\nu_{A,2}$, $\nu_{S,1}$, $\nu_{S,2}$, $\tilde \nu_{S,1}$, $\tilde \nu_{S,2}$, independent of $h$ and of how $\Gamma$ intersects $\mathcal{T}_h$ such that the following spectral inequalities hold:
 \begin{align}
  \nu_{A,1} h^{-1} M_u & \leq A \leq \nu_{A,2} h^{-3} M_u, \label{spec1} \\
  \nu_{S,1} h^{-1} M_\lambda & \leq S \leq \nu_{S,2} h^{-3} M_\lambda, \label{spec2}\\
  \tilde \nu_{S,1}  S_M & \leq S \leq \tilde \nu_{S,2} S_M.  \label{spec3}
 \end{align}
\end{lemma}
\begin{proof}
Note that
\begin{equation}\label{cond1}
\frac{\la A\vec v,\vec v\ra}{\la M_u \vec v, \vec v\ra}=
\frac{A_h(P_h^U \vec v, P_h^U \vec v)}{\|P_h^U \vec v\|^2_{L^2(\OGamma)}}\quad\text{for all}~ \vec v\in\mathbb{R}^n.
  \end{equation}
From \eqref{elliptic} we get
\[ \nu_{A,1}h^{-1}\le\frac{A_h(\bv_h, \bv_h)}{\|\bv_h\|^2_{L^2(\OGamma)}} \quad \text{for all}~\bv_h \in U_h. %\le c_2h^{-3}\quad \forall~\bv_h\in U_h,
\]
Using the local variant of \eqref{fund1B} and a FE inverse inequality we get
\[
  \|\bv_h\|_1^2 \lesssim h^{-1}\|\bv_h\|_{H^1(\OGamma)}^2 + h \sum_{T \in \T_h^\Gamma} |\bv_h|_{H^2(T)}^2 \lesssim h^{-1}\|\bv_h\|_{H^1(\OGamma)}^2 \quad\text{for all}~ \bv_{h} \in \bU_h,
\]
and thus,
\[
 \frac{A_h(\bv_h, \bv_h)}{\|\bv_h\|^2_{L^2(\OGamma)}} \leq c  \frac{\|\bv_h\|_1^2 +\rho \|\bv_h\|_{H^1(\OGamma)}^2}{\|\bv_h\|^2_{L^2(\OGamma)}} \leq c \frac{h^{-1} \|\bv_h\|_{H^1(\OGamma)}^2}{\|\bv_h\|^2_{L^2(\OGamma)}} \leq  \nu_{A,2}h^{-3}~~\text{for all}~\bv_h \in U_h,
\]
with a suitable constant $\nu_{A,2}$.
Combination of these results yields the inequalities in \eqref{spec1}.
For the Schur complement matrix ${S}=B A^{-1} B^T$ we have
\begin{equation}\label{cond3}
\begin{split}
\frac{\la S \vec \lambda,\vec \lambda\ra}{\la S_M\vec \lambda,\vec \lambda\ra}&
=\sup_{\vec u\in\mathbb{R}^n}\frac{\la  \vec u,  A^{-\frac12}B^T \vec\lambda\ra^2}{\|\vec u\|^2\|P_h^M \vec \lambda\|_M^2}
=\sup_{\vec u\in\mathbb{R}^n}\frac{\la B \vec u, \vec \lambda\ra^2}{\| A^\frac12 \vec u\|^2\|P_h^M \vec \lambda\|_\ast^2}\\
&=
\sup_{\bu_h\in U_h}\frac{b(\bu_h,\mu_h )^2}{\|\bu_h\|_U^2\|\mu_h\|_M^2},\quad \mu_h:=P_h^M \vec \lambda.
\end{split}
  \end{equation}
  Using the results in  \eqref{cont2} and Corollary~\ref{corinfsup}  we obtain the result in \eqref{spec3}. We also have
\begin{equation}\label{cond4}
 \frac{\la S_M \vec \lambda,\vec \lambda\ra}{\la M_\lambda\vec \lambda,\vec \lambda\ra} = \frac{\|\mu_h\|_M^2}{\|\mu_h\|_{L^2(\OGamma)}^2}, \quad       \mu_h:=P_h^M \vec \lambda.
\end{equation}
Using $ h \lesssim \rho \lesssim h^{-1}$ and \eqref{fund1} we get
\[
  h^{-1} \|\mu_h\|_{L^2(\OGamma)}^2 \lesssim  \|\mu_h\|_M^2 \lesssim h^{-3} \|\mu_h\|_{L^2(\OGamma)}^2 \quad \text{for all}~\mu_h \in M_h.
\]
Using this we see that the result in  \eqref{spec2} follows from \eqref{spec3}.
\end{proof}
\ \\[1ex]
We introduce a block diagonal preconditioner
\[
  Q:=\left[\begin{matrix}
             Q_A & 0  \\
             0 & Q_S
           \end{matrix}\right]
\]
of $\A$.
% We obtain the following results, in which all spectral equivalences are uniform in $h$ and independent of how $\Gamma$ intersects $\mathcal{T}_h$.
\begin{corollary}\label{corr:pc}
The following estimate holds with some $c>0$ independent of how $\Gamma$ cuts through the background mesh,
\begin{equation}\label{Condest}
\mathrm{cond}(\A)=\|\A\| \|\A^{-1}\|\le c\,h^{-2}.
\end{equation}
\end{corollary}
\begin{proof}
Take $Q_A:=M_u$, $Q_S:=M_\lambda$. Let $\xi$ be an eigenvalue of $Q^{-1}\A$.
We apply the result in Lemma~5.14 from \cite{OlshTyrt} to derive from \eqref{spec1}, \eqref{spec2}:
\begin{equation}\label{EigBound}
\begin{split}
\xi &\in[\nu_{S,1}h^{-1},\nu_{S,2}\,h^{-3}]\cup \left[\frac{\nu_{A,1}+\sqrt{\nu_{A,1}^2+4\nu_{A,1}\nu_{S,1}}}{2h},\frac{\nu_{A,2}+\sqrt{\nu_{A,2}^2+4\nu_{A,2}\nu_{S,2}}}{2h^3}\right]\\
&\cup \left[\frac{\nu_{A,2}-\sqrt{\nu_{A,2}^2+4\nu_{A,2}\nu_{S,2}}}{2h^3},\frac{\nu_{A,1}-\sqrt{\nu_{A,1}^2+4\nu_{A,1}\nu_{S,1}}}{2h}\right].
\end{split}
\end{equation}
From this spectral estimate and the  fact that $Q$ has a uniformly bounded condition number we conclude that \eqref{Condest} holds.
\end{proof}
\ \\[2ex]
Similar to \eqref{EigBound} we also get from \eqref{spec3} the following result.
\begin{corollary}
Let $Q_A \sim A$ be a uniformly spectrally equivalent preconditioner of $A$ and $Q_S:=S_M$. For the spectrum $\sigma(Q^{-1}\A)$
of the preconditioned matrix  we have
\[ \sigma(Q^{-1}\A) \subset \big([C_{-},c_{-}]\cup[c_+,C_+]\big), \]
with some constants
$C_{-} < c_{-} < 0 < c_+ < C_+$ independent of $h$ and the position of $\Gamma$.
\end{corollary}
\ \\[2ex]
Note that the optimal Schur complement preconditioner $S_M$ is easy to implement since the terms occurring in $S_M$ are essentially the same as in the bilinear form $b(\cdot,\cdot)$.  Furthermore, for  $\rho \sim h$ we have the spectral equivalence
\begin{equation}\label{equiv2}
    S_M \sim h^{-1} M_\lambda,
\end{equation}
which follows from \eqref{fund1}. Hence,  systems with the matrix $S_M$ are then easy to solve.

\section{Numerical experiments} \label{sectExp}
In this section we present results of a few numerical experiments. We first consider the vector-Laplace problem \eqref{strongform} on the unit sphere. We use a standard trace-FEM approach in the sense that the exact surface is approximated by a piecewise planar one. Due to this geometric error (${\rm dist}(\Gamma_h,\Gamma) \lesssim h^2$) the discretization accuracy is limited to second order and therefore we consider the discretization \eqref{discrete} with piecewise linears both for the velocity and the Langrange multiplier.  Higher order surface approximations with the technique introduced in \cite{grande2017higher} will be treated in a forthcoming paper. To be able to use a higher order finite element space, and in particular the pair $(V_h^2)^3$-$V_h^1$ for velocity and Lagrange multiplier (which is LBB stable, cf. Corollary~\ref{corinfsup}), we also consider an example in which the surface is a bounded plane which is not aligned with the coordinate axes. For both cases we present results for
discretization errors and their dependence on the stabilization parameter $\rho$. For the problem on the unit sphere we also illustrate the behavior of a preconditioned MINRES solver.

\subsection{Vector-Laplace problem on the unit sphere}
For $\Gamma$ we take the unit sphere, characterized as the zero level of the level set function
$
  \phi(\bx) = \|\bx\|_2 -1$, where $\|\cdot\|_2$ denotes the Euclidean norm on $\R^3$.  We consider the vector-Laplace problem \eqref{strongform} with the prescribed solution $\bu^*=\bP(-x_3^2,x_2,x_1)^T \in\bV_T$. The induced tangential right hand-side is denoted by $\mathbf{f}$ and we use the saddle point formulation \eqref{weak1a} with $\bg = \mathbf{f}$. The associated Lagrange multiplier $\lambda$ is then given by \eqref{charlambda} with $g_N=0$.
The sphere is embedded in an outer domain $\Omega=[-5/3,5/3]^3$. The triangulation $\T_{h_\ell}$ of $\Omega$ consists of $n_\ell^3$ sub-cubes, where each of the sub-cubes is further refined into 6 tetrahedra.   Here $\ell\in\Bbb{N}$ denotes the level of refinement yielding a mesh size $h_\ell= \frac{10/3}{n_\ell}$ with $n_\ell= 2^{\ell+1}$.

On the outer mesh $\T_h^\Gamma$ we define the nodal interpolation operators $I_k:C(\Omega_h^\Gamma) \to V_h^k$. For the TraceFEM, instead of the exact surface $\Gamma$, we consider an approximated surface $\Gamma_h= \{\bx\in\Omega_h^\Gamma:~ (I_1\phi)(\bx)=0\}$, consisting of triangular planar patches. This induces a geometry error with $\mathrm{dist}(\Gamma,\Gamma_h)\lesssim h^2$. The normals $\bn$ are approximated by $\bn_h:=\frac{\nabla\phi_h}{\|\nabla\phi_h\|_2}$,  where $\phi_h$ is the piecewise quadratic interpolant $\phi_h:= I_2(\phi)$.
For the pair of TraceFE spaces $\big((V_h^k)^3,\, V_h^l\big)$ in the saddle point formulation \eqref{discrete}, with $\Gamma$ replaced by $\Gamma_h$, we use  piecewise linear finite elements, i.e.,  $k=l=1$, and we first choose $\rho=\tilde\rho=h$ for the stabilization parameters. The numerical solution on refinement level $\ell=4$ is illustrated in Figure~\ref{fig:solSphere}. The very shape irregular surface triangulation $\Gamma_h$ is illustrated in Figure~\ref{fig:gridSphere}.

\begin{figure}[ht!]
  \begin{minipage}{0.45\textwidth}
    \includegraphics[width=\textwidth]{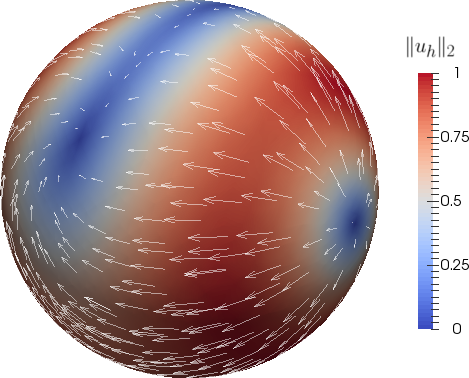}
    \caption{Numerical solution on the sphere for $k=l=1$, $\rho=\tilde\rho=h$ and refinement level $\ell=4$.}
    \label{fig:solSphere}
\end{minipage}
\hfill
%   \end{figure}
%    \begin{table*}
\begin{minipage}{0.45\textwidth}
    \includegraphics[width=\textwidth]{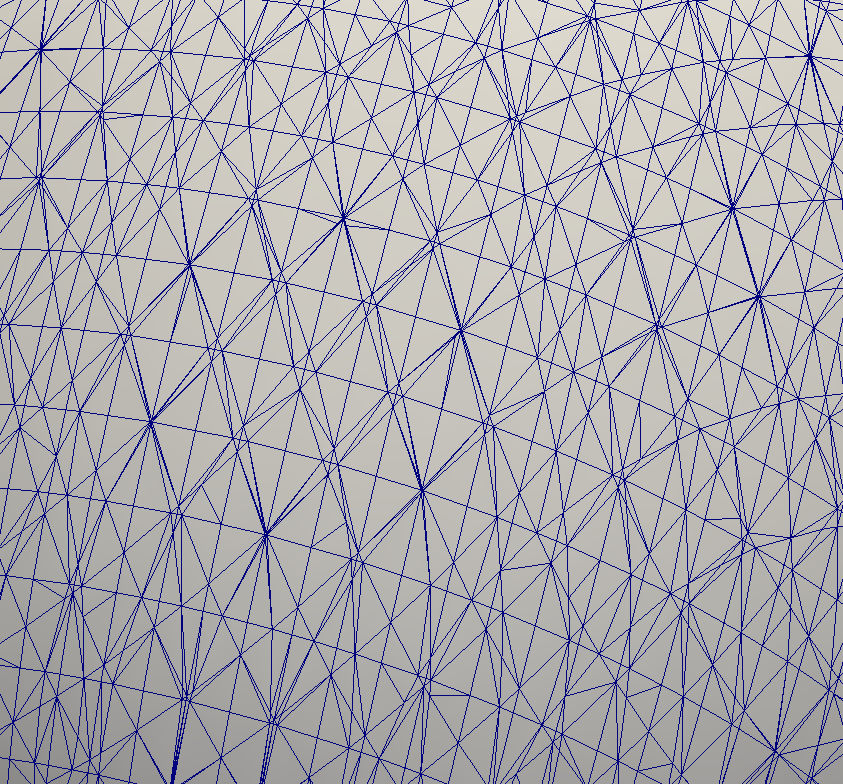}
    \caption{Detail of the surface triangulation $\Gamma_h$ of the sphere for refinement level $\ell=4$.}
    \label{fig:gridSphere}
 \end{minipage}
\end{figure}

Figure~\ref{fig:convSphere1} shows different norms of the errors for different refinement levels $\ell$.

\begin{figure}[ht!]
  \begin{tikzpicture}
  \def\vara{8}
  \def\varb{1.5}
  \begin{semilogyaxis}[ xlabel={Refinement level}, ylabel={Error}, ymin=5E-4, ymax=200, legend style={ cells={anchor=west}, legend pos=outer north east} ]
    \addplot table[x=level, y=U] {sphereP1-h.dat};
    \addplot table[x=level, y=L2P] {sphereP1-h.dat};
    \addplot table[x=level, y=u_N] {sphereP1-h.dat};
    \addplot table[x=level, y=M] {sphereP1-h.dat};
    \addplot[dashed,line width=0.75pt] coordinates { % h
      (1,\vara) (2,\vara*0.5) (3,\vara*0.25)(4,\vara*0.125) (5,\vara*0.0625)
    };
    \addplot[dotted,line width=0.75pt] coordinates { % h^2
      (1,\varb) (2,\varb*0.5*0.5) (3,\varb*0.25*0.25) (4,\varb*0.125*0.125) (5,\varb*0.0625*0.0625)
    };
    \legend{$\|\bu^* - \bu_h\|_U$, $\|\bu^* -\bP\bu_h\|_{L^2(\Gamma_h)}$, $\|\bu_h\cdot\bn_h\|_{L^2(\Gamma_h)}$, $\|\lambda - \lambda_h\|_M$, $\mathcal{O}(h)$, $\mathcal{O}(h^2)$}
  \end{semilogyaxis}
  \end{tikzpicture}
  \caption{Discretization errors  for the sphere and $k=l=1$ with $\rho=\tilde\rho=h$.}
  \label{fig:convSphere1}
\end{figure}
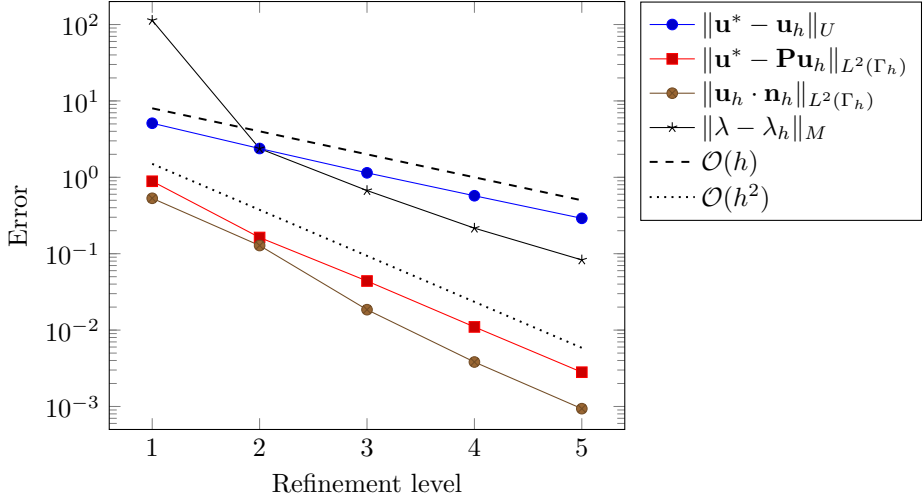

 Optimal convergence orders are achieved for  $\|\bu^*-\bu_h\|_U$ and $\|\bu^* -\bP\bu_h\|_{L^2(\Gamma_h)}$, cf.  the theoretical results in Theorems~\ref{thm2} and \ref{thmL2} (note that in the theoretical analysis we do not treat geometric errors). For $\|\lambda-\lambda_h\|_M$ we observe a convergence order of about 1.5, which is  better than the theoretical result for $\|\lambda-\lambda_h\|_M$. For the normal component of $\bu_h$ we observe  $\|\bu_h\cdot\bn_h\|_{L^2(\Gamma_h)} \sim h^2$, i.e., a faster convergence as predicted by the bound \eqref{eq:normalError}.
From  further experiments (results are not shown) it follows that the same convergence orders are obtained for the choice $\rho=\tilde\rho=1$, but then the errors are slightly increased.

The experiments are repeated for a different choice of the stabilization parameters, taking the other limit case $\rho=\tilde\rho=1/h$. The convergence results are presented in Figure~\ref{fig:convSphere2}.

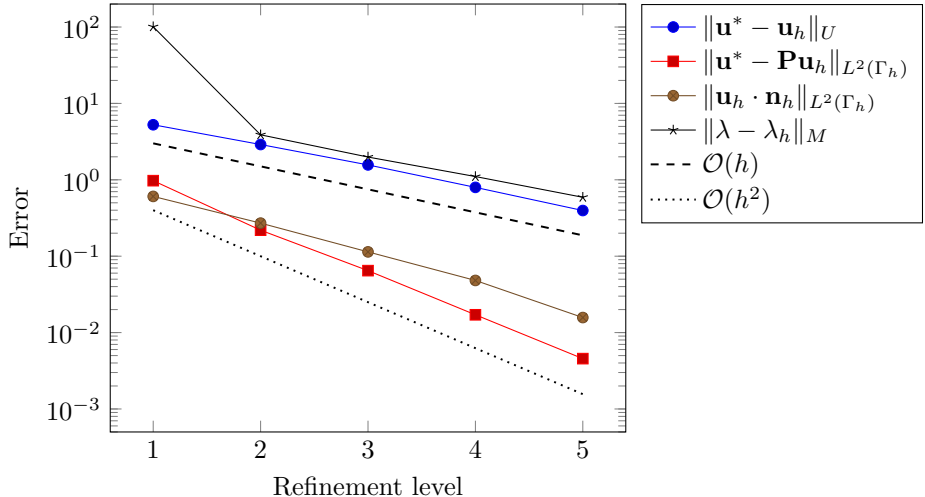
\begin{figure}[ht!]
  \begin{tikzpicture}
  \def\vara{3}
  \def\varb{0.4}
  \begin{semilogyaxis}[ xlabel={Refinement level}, ylabel={Error}, ymin=5E-4, ymax=200, legend style={ cells={anchor=west}, legend pos=outer north east} ]
    \addplot table[x=level, y=U] {sphereP1-hinv.dat};
    \addplot table[x=level, y=L2P] {sphereP1-hinv.dat};
    \addplot table[x=level, y=u_N] {sphereP1-hinv.dat};
    \addplot table[x=level, y=M] {sphereP1-hinv.dat};
    \addplot[dashed,line width=0.75pt] coordinates { % h
      (1,\vara) (2,\vara*0.5) (3,\vara*0.25)(4,\vara*0.125) (5,\vara*0.0625)
    };
    \addplot[dotted,line width=0.75pt] coordinates { % h^2
      (1,\varb) (2,\varb*0.5^2) (3,\varb*0.25^2) (4,\varb*0.125^2) (5,\varb*0.0625^2)
    };
    \legend{$\|\bu^* - \bu_h\|_U$, $\|\bu^* -\bP\bu_h\|_{L^2(\Gamma_h)}$, $\|\bu_h\cdot\bn_h\|_{L^2(\Gamma_h)}$, $\|\lambda - \lambda_h\|_M$, $\mathcal{O}(h)$, $\mathcal{O}(h^{2})$}
  \end{semilogyaxis}
  \end{tikzpicture}
  \caption{Discretization errors for the sphere and $k=l=1$ with $\rho=\tilde\rho=1/h$.}
  \label{fig:convSphere2}
\end{figure}

Compared to Figure~\ref{fig:convSphere1}, the reported errors are all larger, especially those for $\|\lambda-\lambda_h\|_M$ and $\|\bu_h\cdot\bn_h\|_{L^2(\Gamma_h)}$, but the convergence orders still behave in an optimal way, as expected from the error analysis. Only first order convergence is achieved for $\|\lambda-\lambda_h\|_M$ compared to order 1.5 for the choice $\rho=\tilde\rho=h$. For $\|\bu_h\cdot\bn_h\|_{L^2(\Gamma_h)}$ the convergence order drops below 2.

As noted above, due to the geometry error  we do not consider the higher order case $k>1$ here.

Next we study the performance of the iterative solver and the preconditioners. To solve the linear saddle point system \eqref{SLAE} a MINRES solver is applied using the block preconditioner $Q$ presented in Section~\ref{s_cond}. For the application of $\vec  u=Q_A^{-1}\vec b$, the system $A\vec u = \vec b$ is iteratively solved using a standard SSOR-preconditioned CG solver, with a tolerance such that the initial residual is reduced by a factor of $10^4$. The same strategy is used for the Schur complement, i.e., for the application of $ \vec v=Q_S^{-1}\vec c$, the system $S_M\vec v = \vec c$ is iteratively solved using a standard SSOR-preconditioned CG solver, with a tolerance such that the initial residual is reduced by a factor of $10^4$.  We note that these high tolerances are not optimal for the overall efficiency of the preconditioned MINRES solver, but are appropriate to demonstrate important properties of the solver. Furthermore, in practice it is probably more efficient to replace the SSOR
preconditioner for the $A$-block by a more efficient one, e.g., an  ILU or a multigrid method.
Iteration numbers used for the solution of the linear systems are reported in Tables~\ref{tab:iter} and \ref{tab:iter2}. Here $N$ denotes the number of MINRES iterations needed to reduce the residual by a factor of $10^6$ (initial vector $\vec x_0=0$). $N_A$ and $N_S$ denote average PCG iteration numbers (per MINRES iteration) needed to apply the preconditioners $Q_A$ and $Q_S$, respectively.

\begin{table}[ht!]
  \begin{minipage}{0.45\textwidth}
  \centering
    \begin{tabular}{l r r r}
      \toprule
      level $\ell$ & $N_A$ & $N_S$ & $N$ \\ \midrule
	1	& 8.2&	6.7	& 45\\
	2	& 17.4&	7.0	& 29\\
	3	& 35.4&	7.3	& 29\\
	4	& 69.0&	8.4	& 28\\
	5	& 138.6&	9.1	& 28\\      \bottomrule
    \end{tabular}
    \caption{Average PCG iteration numbers ($N_A, N_S$) for application of $Q_A^{-1}, Q_S^{-1}$, respectively, and MINRES iteration numbers ($N$) for different refinement levels and $\rho=\tilde\rho=h$.}
    \label{tab:iter}
 \end{minipage}
 \hfill
\begin{minipage}{0.45\textwidth}
  \centering
    \begin{tabular}{l r r r}
      \toprule
      level $\ell$ & $N_A$ & $N_S$ & $N$ \\ \midrule
	1	& 8.0&	7.5	& 64\\
	2	& 16.7 & 13.5	& 38\\
	3	& 33.0 & 23.3	& 34\\
	4	& 64.2 & 46.1	& 32\\
	5	& 126.4& 89.9	& 29\\      \bottomrule
    \end{tabular}
    \caption{Average PCG iteration numbers ($N_A, N_S$) for application of $Q_A^{-1}, Q_S^{-1}$, respectively, and MINRES iteration numbers ($N$) for different refinement levels and $\rho=\tilde\rho=1/h$.}
    \label{tab:iter2}
 \end{minipage}
\end{table}

We first discuss the case $\rho=\tilde\rho=h$, cf. Table~\ref{tab:iter}. The number of MINRES iterations does not grow if the refinement level $\ell$ increases, illustrating the optimality of the block preconditioner, cf. Corollary~\ref{corr:pc}.  As expected, cf. \eqref{equiv2}, the numbers $N_S$ are essentially independent of $\ell$. The numbers $N_A$ are doubled for each refinement and show a behavior very similar to the usual behavior of SSOR-CG applied to a standard 2D Poisson problem (discretized by linear finite elements). In this paper we do not study the topic of more efficient preconditioners for $Q_A$.

For the case $\rho=\tilde\rho=1/h$, cf. Table~\ref{tab:iter2}, the average iteration number $N_S$ of the Schur complement preconditioner shows a similar behavior as $N_A$, i.e., a doubling of the iteration numbers for each refinement step. These iteration numbers $N_S$ indicate $\mathrm{cond}(S_M)\sim h^{-2}$. Note that for \eqref{equiv2} to hold we used a scaling assumption $\rho=\tilde\rho=h$.    We also observe that (slightly) more MINRES iterations $N$ are needed compared to the case $\rho=\tilde\rho=h$.

In view of the  results for the discretization errors and the results for the iterative solver obtained in these numerical experiments the parameter choice $\rho=\tilde\rho=h$ is recommended.
\medskip

\subsection{Vector-Laplace problem on a planar surface} \label{sec:numExpPlane}
Let $G$ be the plane defined by  zero level of the level set function
$
  \phi(\bx) = -x_3+4x_1-\frac{13}{3}$ and $\Omega \subset \Bbb{R}^3$ a bounded outer domain that is intersected by $G$. We take $\Gamma:=G \cap \Omega$.
 We consider the vector-Laplace problem \eqref{strongform} with the prescribed solution
\begin{equation*}
\bu^*(\mathbf{x}) = (\frac{4}{17} \sin(\pi x_2)\sin(\pi x_3), \sin(\pi x_2)\sin(\pi x_3), \frac{16}{17}\sin(\pi x_2)\sin(\pi x_3))^T \in \bV_T.
\end{equation*}
The induced tangential right hand-side is denoted by $\mathbf{f}$ and we use the saddle point formulation \eqref{weak1a} with $\bg = \mathbf{f}$. The associated Lagrange multiplier is $\lambda=0$.
Concerning the choice of the outer domain $\Omega$ there is an issue concerning Dirichlet boundary conditions on $\partial\Omega \cap G=\partial \Gamma$. It turns out that we obtain significantly larger  discretization errors if the parts of $\partial \Omega$ that are intersected by $G$ are not perpendicular to $G$. This effect is not understood, yet. Therefore in this experiment  we choose $\Omega$ as the parallelepiped with origin $(-2,-2,-65/48)$ and spanned by the vectors $(4,0,-1)^T$, $(0,4,0)^T$ and $(0,0,4.25)^T$.   Then the parts of $\partial \Omega$ that are intersected by $G$ are  perpendicular to $G$. The construction is such that on $\partial\Omega \cap G=\partial \Gamma$ we can use homogeneous  Dirichlet boundary conditions.

The triangulation $\T_{h_\ell}$ of $\Omega$ consists of $n_\ell^3$ sub-parallelepipeds, where each of the sub-parallelepipeds is further refined into 6 tetrahedra. Here $\ell\in\Bbb{N}$
denotes the level of refinement yielding a mesh size $h_\ell= \frac{4}{n_\ell}$ with $n_\ell= 2^{\ell+2}$. Note that for this specific example, the approximation of the surface $\Gamma$ and the normals $\bn$ is exact, i.e., $\Gamma_h=\Gamma$ and $\bn_h=\bn$.

The numerical solution on refinement level $\ell=4$ is illustrated in Figure~\ref{fig:solPlane}. The (very shape irregular)  surface triangulation is illustrated in Figure~\ref{fig:gridPlane}.

\begin{figure}[ht!]
  \begin{minipage}{0.45\textwidth}
    \includegraphics[width=\textwidth]{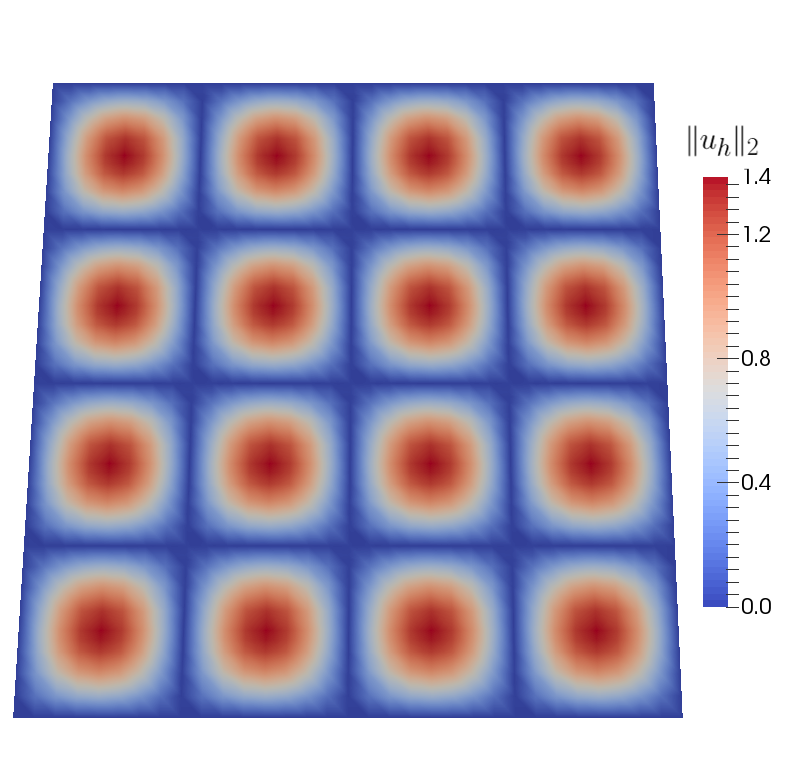}
    \caption{Numerical solution on the plane for $k=l=1$, $\rho=\tilde\rho=h$ and refinement level $\ell=3$.}
    \label{fig:solPlane}
\end{minipage}
\hfill
%   \end{figure}
%    \begin{table*}
\begin{minipage}{0.45\textwidth}
    \includegraphics[width=\textwidth]{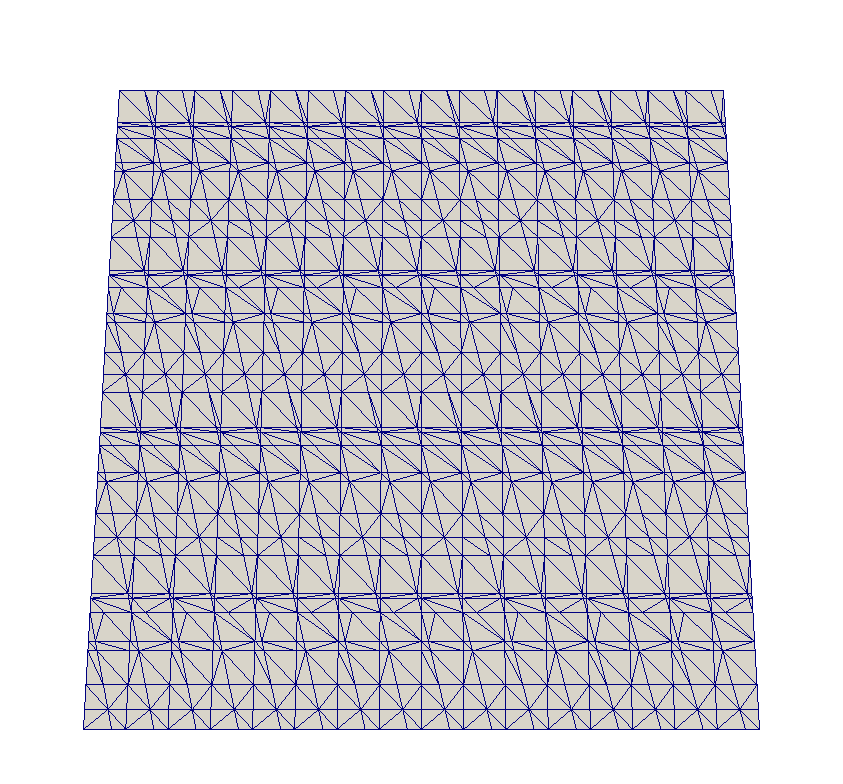}
    \caption{Surface triangulation for refinement level $\ell=1$.}
    \label{fig:gridPlane}
 \end{minipage}
\end{figure}

%In order to ensure that the plane lies perpendicular to all boundaries of the outer domain $\Omega$,
% \textbf{Why do we need this, TL? TJ: At first I tried $\Omega = [-2,2]^3$, which resulted in optimal convergence orders for $\rho = \tilde{\rho} =h$ but not for $\rho= \tilde{\rho} =1/h$. For $\rho= \tilde{\rho} =1/h$ I had bigger errors close the boundaries, to which the surface was not perpendicular. This problem does not occur when the surface lies perpendicular to all boundaries of the outer domain. In the case of $\rho= \tilde{\rho} =1/h$ we assign a larger weight to the volume normal derivative stabilization term. I don't have an explanation yet why this results in bigger errors.}
%we choose $\Omega$ as the parallelepiped with origin $(-2,-2,-65/48)^T$ and spanned by the vectors $(4,0,-1)^T$, $(0,4,0)^T$ and $(0,0,4.25)^T$.
%The right hand side $\bg$ in \eqref{weak1a} is chosen such that
%\begin{equation*}
%\bu^*(\mathbf{x}) = (\frac{4}{17} \sin(\pi x_2)\sin(\pi x_3), \sin(\pi x_2)\sin(\pi x_3), \frac{16}{17}\sin(\pi x_2)\sin(\pi x_3))^T \in \bV_T
%\end{equation*}
%solves \eqref{strongform} with $\mathbf{f} = \bg$. Homogeneous Dirichlet boundary conditions are applied on the boundary of $\Omega$.

Note that $\lambda=\lambda_h=0$ due to $\mathbf{H}=0$, cf. \eqref{charlambda}, so in the following we only consider errors in $\bu_h$.
In view of the recommendation at the end of the section above, we only present  results with the stabilization parameter $\rho = \tilde{\rho} = h$.
% as choosing $h \lesssim \rho = \tilde{\rho} \lesssim h^{-1}$ leads to very similar results.
Figures \ref{fig:convPlane1} and \ref{fig:convPlane2} show the errors $\|\bu^*-\bu_h\|_U$ and $\|\bu^* - \bP\bu_h\|_{L^2(\Gamma)}$ for the cases $k=l=1$ and $k=2,\, l=1$, respectively. In both cases, optimal convergence orders $\|\bu^*-\bu_h\|_U \sim h^k$ and $\|\bu^* -\bP\bu_h\|_{L^2(\Gamma)} \sim h^{k+1}$ are achieved.  In this special planar setting we have $\bu_h\cdot\bn_h=0$.
% We start with the case $k=l=1$. Figure \ref{fig:convPlane1} shows the norms of the errors for the different refinement levels $\ell$.  We observe optimal convergence order for both norms. The second case we consider is $k=2$ and $l=1$. The $U$-norm and the $\|\bu^* -\bP\bu_h\|_{L^2(\Gamma_h)}$-norm discretization error are given in figure \ref{fig:convPlane2}. These results clearly show that the $U$-norm scales like $\mathcal{O}(h^2)$ and the $\|\bu^* -\bP\bu_h\|_{L^2(\Gamma_h)}$-norm scales like $\mathcal{O}(h^3)$, which is optimal in both cases.    \textbf{Bezug zu Laplace in 2D, da kein Rand?}

\begin{figure}[ht!]
  \begin{tikzpicture}
  \def\vara{8}
  \def\varb{2}
  \begin{semilogyaxis}[ xlabel={Refinement level}, ylabel={Error}, ymin=5E-4, ymax=200, legend style={ cells={anchor=west}, legend pos=outer north east} ]
    \addplot table[x=level, y=U] {PlaneP1.dat};
    \addplot table[x=level, y=L2] {PlaneP1.dat};
    \addplot[dashed,line width=0.75pt] coordinates { % h
      (1,\vara) (2,\vara*0.5) (3,\vara*0.25)(4,\vara*0.125) (5,\vara*0.0625)
    };
    \addplot[dotted,line width=0.75pt] coordinates { % h^2
      (1,\varb) (2,\varb*0.5*0.5) (3,\varb*0.25*0.25) (4,\varb*0.125*0.125) (5,\varb*0.0625*0.0625)
    };
    \legend{$\|\bu^* - \bu_h\|_U$, $\|\bu^* -\bP\bu_h\|_{L^2(\Gamma)}$, $\mathcal{O}(h)$, $\mathcal{O}(h^2)$}
  \end{semilogyaxis}
  \end{tikzpicture}
  \caption{Discretization errors for planar $\Gamma$, $k=l=1$ and $\rho=\tilde\rho=h$.}
  \label{fig:convPlane1}
\end{figure}

\begin{figure}
  \begin{tikzpicture}
  \def\vara{1.5}
  \def\varb{0.3}
  \begin{semilogyaxis}[ xlabel={Refinement level}, ylabel={Error}, ymin=5E-6, ymax=100, legend style={ cells={anchor=west}, legend pos=outer north east} ]
    \addplot table[x=level, y=U] {PlaneP2.dat};
    \addplot table[x=level, y=L2] {PlaneP2.dat};
    \addplot[dashed,line width=0.75pt] coordinates { % h^2
      (1,\vara) (2,\vara*0.5*0.5) (3,\vara*0.25*0.25) (4,\vara*0.125*0.125) (5,\vara*0.0625*0.0625)
    };
    \addplot[dotted,line width=0.75pt] coordinates { % h^3
      (1,\varb) (2,\varb*0.5*0.5*0.5) (3,\varb*0.25*0.25*0.25) (4,\varb*0.125*0.125*0.125) (5,\varb*0.0625*0.0625*0.0625)
    };
    \legend{$\|\bu^* - \bu_h\|_U$, $\|\bu^* -\bP\bu_h\|_{L^2(\Gamma)}$, $\mathcal{O}(h^2)$, $\mathcal{O}(h^{3})$}
  \end{semilogyaxis}
  \end{tikzpicture}
  \caption{Discretization errors for planar $\Gamma$,  $k=2$, $l=1$ and $\rho=\tilde\rho=h$.}
  \label{fig:convPlane2}
\end{figure}

\bibliographystyle{siam}
\bibliography{literatur}{}

\end{document}